\documentclass{article}
\usepackage{amsthm,amsmath,amsfonts,amssymb, hyperref,authblk}
\usepackage{vmargin}
\setmarginsrb{2cm}{1cm}{2cm}{3cm}{1cm}{1cm}{2cm}{2cm}

\usepackage[english]{babel}
\usepackage[T1]{fontenc}
\usepackage{lmodern}
\usepackage{amsfonts}

\usepackage{tkz-euclide,subfigure}
\usetkzobj{all}

 \usepackage{pgf,tikz}
\usetikzlibrary{arrows}

\usepackage{amssymb,mathtools}
\usepackage{esdiff}
\usepackage{array}
\usepackage{siunitx}
\usepackage{scrtime}
\usepackage{mwe} 
 \usepackage{pgf,tikz}
\usetikzlibrary{arrows}
\usepackage{stmaryrd}
\usepackage{xparse}
\usepackage{cases}
\usepackage{booktabs}

\newcommand{\N}{\mathbf{N}}
\newcommand{\R}{\mathbf{R}}

\newcommand{\T}{\mathbb{T}}

\newcommand{\ens}[2]{\left\{\, #1\ |\  #2\, \right\}}
 \newcommand{\p}[1]{\partial_{#1}}
 \renewcommand{\d}{\mathrm{d}}
 \newcommand{\ps}[2]{\langle #1,#2\rangle}
\newcommand{\Ps}[2]{\bigl\langle #1,#2\bigr\rangle}

\newcommand{\af}{\alpha_f}
\renewcommand{\bf}{\beta_f}
 \newcommand{\ap}{\alpha_p}
\newcommand{\bp}{\beta_p}

\newcommand{\Hbc}{\mathcal{H}^1_{\mathrm{bc}}}
\newcommand{\tHbc}{\tilde{\mathcal{H}}^1_{\mathrm{bc}}}

\DeclareMathOperator{\Id}{I}
\DeclareMathOperator{\ran}{Ran}
\DeclareMathOperator{\Ker}{Ker}

\newcolumntype{L}[1]{>{\raggedright\arraybackslash}p{#1}}

\newtheorem{thm}{Theorem}[section]
\newtheorem{lemma}[thm]{Lemma}
\newtheorem{proposition}[thm]{Proposition}

\newtheorem{e-definition}[thm]{Definition}
\newtheorem{remark}[thm]{\textbf Remark\/}

\numberwithin{equation}{section}

\makeatletter
\newcommand{\@syst}[1]{%
\[ 
\left\{
\begin{aligned}
#1
\end{aligned}
\right.
\]
}
\newenvironment{syst}{\collect@body\@syst}{\global\@ignoretrue}

\newcommand{\@systn}[1]{%
\begin{equation}
\left\{
\begin{aligned}
#1
\end{aligned}
\right.
\end{equation}
}
\newenvironment{systn}{\collect@body\@systn}{\global\@ignoretrue}
\makeatother

\title{latexdiff Example - \DIFdelbegin \DIFdel{Draft }\DIFdelend \DIFaddbegin \DIFadd{Revised }\DIFaddend version}
\RequirePackage[normalem]{ulem} 
\RequirePackage{color}\definecolor{RED}{rgb}{1,0,0}\definecolor{BLUE}{rgb}{0,0,1} 
\providecommand{\DIFadd}[1]{{\protect\color{blue}\uwave{#1}}} 
\providecommand{\DIFdel}[1]{{\protect\color{red}\sout{#1}}}                      
\providecommand{\DIFaddbegin}{} 
\providecommand{\DIFaddend}{} 
\providecommand{\DIFdelbegin}{} 
\providecommand{\DIFdelend}{} 

\newcolumntype{L}[1]{>{\raggedright\arraybackslash}p{#1}}

\title{Analysis and Output Tracking Design for the Direct Contact Membrane Distillation Parabolic System.}

\author[1]{Mohamed Ghattassi}
\author[1]{Taous-Meriem Laleg}
\author[2]{Jean-Claude Vivalda}
\affil[1]{CEMSE Division, King Abdullah University of Science and Technology (KAUST), Thuwal, 23955-6900, Kingdom of Saudi Arabia, {\sf mohamed.ghattassi@kaust.edu.sa} {\sf taousmeriem.laleg@kaust.edu.sa}}
\affil[2]{  Inria, Villers-l\`es-Nancy, F-54600, France-- Universit\'e de Lorraine, IECL, UMR 7502, Vand\oe uvre-l\`es-Nancy, F-54506, France, {\sf Jean-Claude.Vivalda@inria.fr}}

\date{}

\begin{document}

\maketitle
\abstract{
This paper considers the  performance output tracking for a boundary controlled Direct Contact Membrane Distillation (DCMD) system. First, the mathematical properties of a recently developed mathematical  model of the DCMD system are discussed. This model consists of parabolic equations  coupled at the boundary. Then, the existence and uniqueness of  the solutions are analyzed, using the theory of operators. Some regularity results of the solution are also established. A particular case showing the diagonal property of the principal operator is studied. Then, based on one-side feedback law the control problem, which consists of  tracking both the feed and permeate outlet temperatures of the membrane distillation system is formulated. A servomechanism and an output feedback controller are proposed to solve the control problem. In addition,  an extended state observer aimed at estimating both the system state and disturbance, based on the temperature measurements of the inlet is proposed.  Thus, by some regularity for the reference signal and when the disturbance vanishes, we prove the exponential decay of the output tracking error. Moreover, we show the performance of the control strategy in presence of the flux noise.
\vspace{0.5cm}


\section{Introduction}

The access to drinking water is getting more and more challenging as a result of the limited natural freshwater resources. On the other hand, the demand for the potable water is increasing, due to the rapid population growth  and  the effects of climate change. As a result, many countries rely on desalination respond to their demand in potable water. Indeed, desalination has been recognized as one of the most promising methods  to reduce water shortage in arid regions through the production of fresh water from seawater and saline groundwater.  Among conventional  water desalination technologies is membrane distillation (MD), which has   great potential for  sustainable high quality water supply. It consists of a separation process driven by temperature gradient, where hot salt water is circulated in one side (the feed side) of a hydrophobic porous membrane, while cold-fresh stream is circulated in the other side (the permeate side),  thus creating a difference of  pressure between the two sides of the membrane that constitutes the main driving force of the process. There are different configurations for  MD systems such as Direct Contact Membrane Distillation (DCMD), Air Gap Membrane Distillation (AGMD), Sweeping Gas Membrane Distillation (SGMD) and Vacuum Membrane Distillation (VMD). More details on the MD technology and its configurations are provided in\cite{khayet_book,lee2017total,el2006framework}. In the past few years, many studies have been conducted  by engineers to propose accurate mathematical models of MD systems and to develop efficient model-based control and monitoring strategies \cite{karam2017analysis, shim2015solar, naidu2017transport}.  An  accurate mathematical model can help  optimize  the system and increase its efficiency.  Among the proposed MD models we will focus on a system of two-dimensional advection diffusion equations coupled at the boundary;  this dynamical model has been proposed for the DCMD configuration and  has been validated experimentally in \cite{eleiwi2016dynamic}. Moreover, based on external observation, Boumenir et al.\cite{boumenir2019monitoring}, propose a method to recover the temperature of the membrane, which is considered as an unknown source term in a parabolic system.

The aim of this paper is to study the mathematical properties of the parabolic system modeling the DCMD   and  to design a controller to track the output temperature of this parabolic system based on some boundary measurements. Our mathematical analysis is performed in the framework of  semigroup theory. Using some classical arguments in the analysis of partial differential equations, we show that the model operator is   m-dissipative. In particular, we show that this operator is diagonal
in the co-current configuration, if some additional conditions on thermal conductivity
and flow rate are satisfied.

Additionally, we demonstrate that, for  any initial conditions, the solution of the system tends to an equilibrium as  time tends to infinity. It is worth noting that systems of advection-diffusion equations  represent an important class of PDEs that  arise in many problems of science and engineering. In this context, there exist some papers that have been devoted to the study of the reaction-advection-diffusion systems for linear and nonlinear cases, see \cite{cieslak2010finite, mizoguchi2014nondegeneracy,corrias2004global}. In these studies, the authors  devote particular attention to the the well-posedness and the blow-up of the solution for a class of nonlinear reaction-advection-diffusion system with internal coupling. 

In this work, for the control design,  we develop  an output feedback  strategy  to track the feed and permeate outlet temperatures of the 2D parabolic system for arbitrary reference signal, using some boundary measurements.  There are several control strategies for  parabolic systems  in the multidimensional case, for instance, the  backstepping method which has been proposed in  \cite{meurer2012control, meurer2009trajectory}. However,  most existing approaches requires the operator to be self-adjoint. But  the operator of the studied parabolic system  is self-adjoint only for some particular values of thermal conductivity and flow rate coefficients. This is why, we propose  to use  the active disturbance rejection control (ADRC), which only needs an analytic semigroup operator and allows the handling of external disturbances. 

In this paper, we introduce an extension of the ADRC method for the  output temperature  tracking of the system of the parabolic equation, weakly coupled at the boundary and subject to external disturbances. We  propose a one-side  feedback law  to track a desired outlet temperature of the  DCMD system. To this end, the main idea is to provide a useful estimation of the disturbance, that is incorporated in the control law to allow an efficient decoupling of the actual disturbance. This  topic has been well-documented for finite dimensional systems, see  for example,  \cite{francis1977linear, desoer1985tracking}. There are also other studies  that investigate the ADRC for PDEs. For example, an ADRC controller has been derived for the wave equation \cite{ guo2016performance, guo2017adaptive}; other studies explore applying the ADRC to parabolic equations \cite{feng2017new,jin2018performance}. In \cite{ feng2017new}, ADRC is combined to the  backstepping method to stabilize the unstable one dimensional heat equation with an external disturbance and boundary uncertainty. In \cite{jin2018performance}, ARDC  is  considered for the boundary output tracking for a one dimensional heat equation with external disturbance. In that paper, the authors propose a design of an observer to estimate the disturbance and a servo  system consisting of an output feedback boundary control law that has been derived. The output tracking design result in this work can be treated as an extension from 1D heat equation considered in \cite{jin2018performance} to 2D parabolic equations weakly coupled at the boundary and as an application for the heat transfer in DCMD system.

This paper is organized as follows: Section~\ref{section1} describes  a  mathematical  model for the heat transfer in DCMD systems. The proof of the existence and uniqueness for the solution of the DCMD elliptic system is established in section~\ref{sec-exist-unic}; we formulate the problem using the framework of operator theory and show that the operator related to the DCMD system is m-dissipative in Appendix ~\ref{deft-op-A}. In section~\ref{cocurentsection} the  co-current DCMD case is presented  and it is shown that  under some additional conditions  the operator is diagonal. In section~\ref{section6}, the output-tracking problem for the DCMD system is formulated, and a control solution based on the ADRC technique is proposed. A numerical example is presented in section~\ref{sectionnumeric} to demonstrate the effectiveness of the proposed control. Finally, in section~\ref{section7}, we conclude by discussing open questions and future studies.

\section{Mathematical modeling of heat transfer in DCMD process}\label{section1}
The model geometry consists of a feed inlet boundary $\mathrm{B}_1$, feed outlet boundary $\mathrm{B}_3$,
 permeate inlet boundary $\mathrm{B}_4$, permeate outlet boundary $\mathrm{B}_6$. In this module, the vapor generated in the feed solution (warm sea water) is forced to pass through the membrane dry pores to the permeate side (cold water), following thermodynamics rules. Hereafter, we outline the equations describing the evolutions of the temperatures in the feed and permeate rooms of the devices: more details can be found in~\cite{eleiwi2016dynamic} or~\cite{shim2015solar}.

  We  denote by $f(t,x,y)$ the temperature of the warm water and by $p(t,x,y)$ the temperature of the cold water at time $t$ and at the point of coordinates $(x,y)$; we denote also by  $ \Omega_f$ and $ \Omega_p$ the rectangles
   $[0,\ell]\times[0,L]$ and $[ \delta_m+\ell,\delta_m+2\ell] \times[0,L]$ respectively  (here $ \delta_m$ denotes the membrane thickness, see Fig.~\ref{dessin}). 

\begin{figure}[!h]
\definecolor{cqcqcq}{rgb}{0.75,0.75,0.75}
\definecolor{qqqqff}{rgb}{0,0,1}
\definecolor{ffqqqq}{rgb}{1,0,0}
\centering
\subfigure[Counter-current presentation]{
\begin{tikzpicture}[line cap=round,line join=round,>=triangle 45,x=1.0cm,y=1.0cm]
\clip(-0.68,-3.76) rectangle (8.58,4.82);
\fill[color=ffqqqq,fill=ffqqqq,fill opacity=0.15] (1.2,3.1) -- (3.2,3.1) -- (3.2,-2.22) -- (1.2,-2.22) -- cycle;
\fill[color=qqqqff,fill=qqqqff,fill opacity=0.15] (4,3.1) -- (6,3.1) -- (6,-2.22) -- (4,-2.22) -- cycle;
\fill[color=cqcqcq,fill=cqcqcq,fill opacity=0.4] (3.2,3.1) -- (4,3.1) -- (4,-2.22) -- (3.2,-2.22) -- cycle;
\draw (1.2,3.1)-- (1.2,-2.22);
\draw (3.2,3.1)-- (3.2,-2.22);
\draw (4,3.1)-- (4,-2.22);
\draw (6,3.1)-- (6,-2.22);
\draw (4,3.1)-- (6,3.1);
\draw (4,-2.22)-- (6,-2.22);
\draw (1.2,3.1)-- (3.2,3.1);
\draw (1.2,-2.22)-- (3.2,-2.22);
\draw (3.2,-2.22)-- (4,-2.22);
\draw (3.2,3.1)-- (4,3.1);
\draw [->,line width=1.2pt] (2.2,-1.08) -- (2.2,1.82);
\draw [->,line width=1.2pt] (5,1.82) -- (5,-1.08);
\draw (1.2,-3)-- (3.2,-3);
\draw (3.2,-3)-- (4,-3);
\draw (4,-3)-- (6,-3);
\draw (1.2,-3.1)-- (1.2,-2.9);
\draw (3.2,-2.9)-- (3.2,-3.1);
\draw (4,-3.1)-- (4,-2.9);
\draw (6,-3.1)-- (6,-2.9);
\draw [->,line width=0.4pt,dash pattern=on 2pt off 2pt] (6,-3) -- (8,-3);
\draw (0.4,3.1)-- (0.4,-2.22);
\draw (0.3,3.1)-- (0.5,3.1);
\draw (0.3,-2.22)-- (0.5,-2.22);
\draw [->,line width=0.4pt,dash pattern=on 2pt off 2pt] (0.4,3.1) -- (0.4,4.6);
\draw (2.02,-2.96) node[anchor=north west] {$\ell$};
\draw (3.39,-2.96) node[anchor=north west] {$ \delta_m$};
\draw (4.76,-2.96) node[anchor=north west] {$\ell$};
\draw (7.88,-2.96) node[anchor=north west] {$x$};
\draw (-0.12,1.38) node[anchor=north west] {$L$};
\draw (-0.12,4.88) node[anchor=north west] {$y$};
\draw (2.02,-2.2) node[anchor=north west] {$\mathrm{B}_1$};
\draw (0.55,0.76) node[anchor=north west] {$\mathrm{B}_2$};
\draw (2.02,3.8) node[anchor=north west] {$\mathrm{B}_3$};
\draw (4.78,3.8) node[anchor=north west] {$\mathrm{B}_4$};
\draw (6.04,0.8) node[anchor=north west] {$\mathrm{B}_5$};
\draw (4.74,-2.2) node[anchor=north west] {$\mathrm{B}_6$};
\draw (3.3,1.34) node[anchor=north west] {$\rotatebox{90.0}{ \text{ MEMBRANE }  }$};
\draw (2.7,0.8) node[anchor=north west] {$\mathrm{I_f}$};
\draw (4.08,0.8) node[anchor=north west] {$\mathrm{I_p}$};
\end{tikzpicture}
}
\hfill
\subfigure[Co-current presentation]{
\begin{tikzpicture}[line cap=round,line join=round,>=triangle 45,x=1.0cm,y=1.0cm]
\clip(-0.68,-3.76) rectangle (8.58,4.82);
\fill[color=ffqqqq,fill=ffqqqq,fill opacity=0.15] (1.2,3.1) -- (3.2,3.1) -- (3.2,-2.22) -- (1.2,-2.22) -- cycle;
\fill[color=qqqqff,fill=qqqqff,fill opacity=0.15] (4,3.1) -- (6,3.1) -- (6,-2.22) -- (4,-2.22) -- cycle;
\fill[color=cqcqcq,fill=cqcqcq,fill opacity=0.4] (3.2,3.1) -- (4,3.1) -- (4,-2.22) -- (3.2,-2.22) -- cycle;
\draw (1.2,3.1)-- (1.2,-2.22);
\draw (3.2,3.1)-- (3.2,-2.22);
\draw (4,3.1)-- (4,-2.22);
\draw (6,3.1)-- (6,-2.22);
\draw (4,3.1)-- (6,3.1);
\draw (4,-2.22)-- (6,-2.22);
\draw (1.2,3.1)-- (3.2,3.1);
\draw (1.2,-2.22)-- (3.2,-2.22);
\draw (3.2,-2.22)-- (4,-2.22);
\draw (3.2,3.1)-- (4,3.1);
\draw [->,line width=1.2pt] (2.2,-1.08) -- (2.2,1.82);
\draw [->,line width=1.2pt] (5,-1.22) -- (5,1.68);
\draw (1.2,-3)-- (3.2,-3);
\draw (3.2,-3)-- (4,-3);
\draw (4,-3)-- (6,-3);
\draw (1.2,-3.1)-- (1.2,-2.9);
\draw (3.2,-2.9)-- (3.2,-3.1);
\draw (4,-3.1)-- (4,-2.9);
\draw (6,-3.1)-- (6,-2.9);
\draw [->,line width=0.4pt,dash pattern=on 2pt off 2pt] (6,-3) -- (8,-3);
\draw (0.4,3.1)-- (0.4,-2.22);
\draw (0.3,3.1)-- (0.5,3.1);
\draw (0.3,-2.22)-- (0.5,-2.22);
\draw [->,line width=0.4pt,dash pattern=on 2pt off 2pt] (0.4,3.1) -- (0.4,4.6);
\draw (2.02,-2.96) node[anchor=north west] {$\ell$};
\draw (3.39,-2.96) node[anchor=north west] {$ \delta_m$};
\draw (4.76,-2.96) node[anchor=north west] {$\ell$};
\draw (7.88,-2.96) node[anchor=north west] {$x$};
\draw (-0.12,1.38) node[anchor=north west] {$L$};
\draw (-0.12,4.88) node[anchor=north west] {$y$};
\draw (2.02,-2.2) node[anchor=north west] {$\mathrm{B}_1$};
\draw (0.55,0.76) node[anchor=north west] {$\mathrm{B}_2$};
\draw (2.02,3.8) node[anchor=north west] {$\mathrm{B}_3$};
\draw (4.78,3.8) node[anchor=north west] {$\mathrm{B}_6$};
\draw (6.04,0.8) node[anchor=north west] {$\mathrm{B}_5$};
\draw (4.74,-2.2) node[anchor=north west] {$\mathrm{B}_4$};
\draw (3.3,1.34) node[anchor=north west] {$\rotatebox{90.0}{ \text{ MEMBRANE }  }$};
\draw (2.7,0.8) node[anchor=north west] {$\mathrm{I_f}$};
\draw (4.08,0.8) node[anchor=north west] {$\mathrm{I_p}$};
\end{tikzpicture}}
\caption{Schematic of the device}
 \label{dessin}
\end{figure}
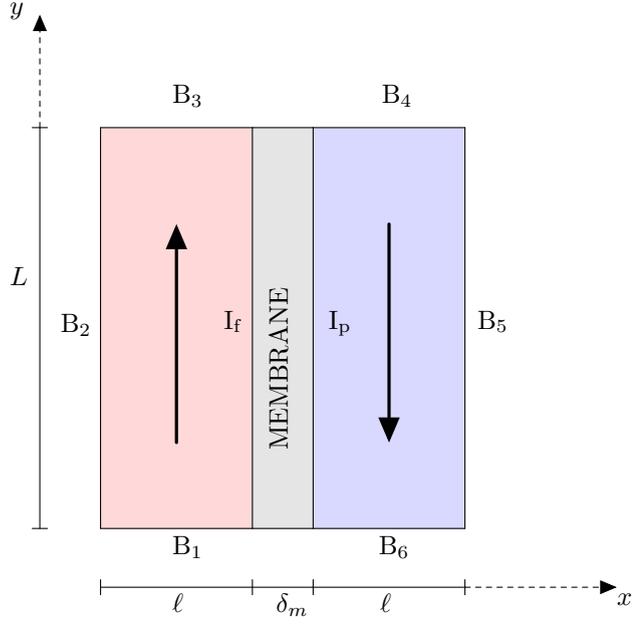
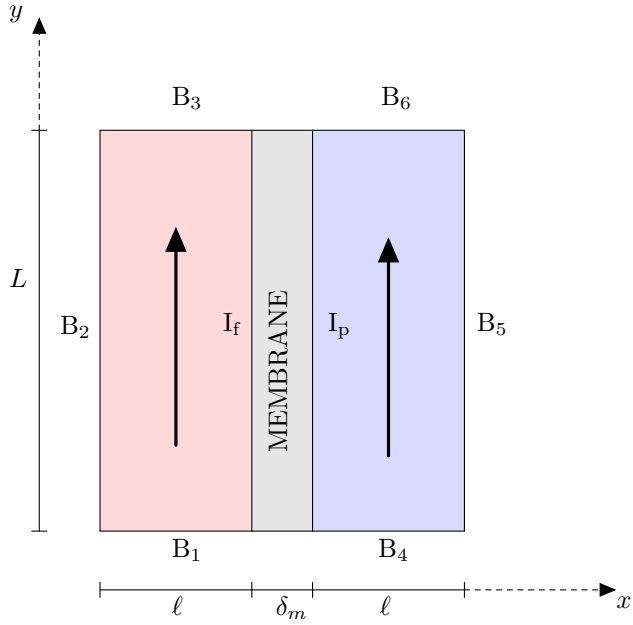

 \begin{figure}[ht]
\centering
{\includegraphics[scale=0.35]{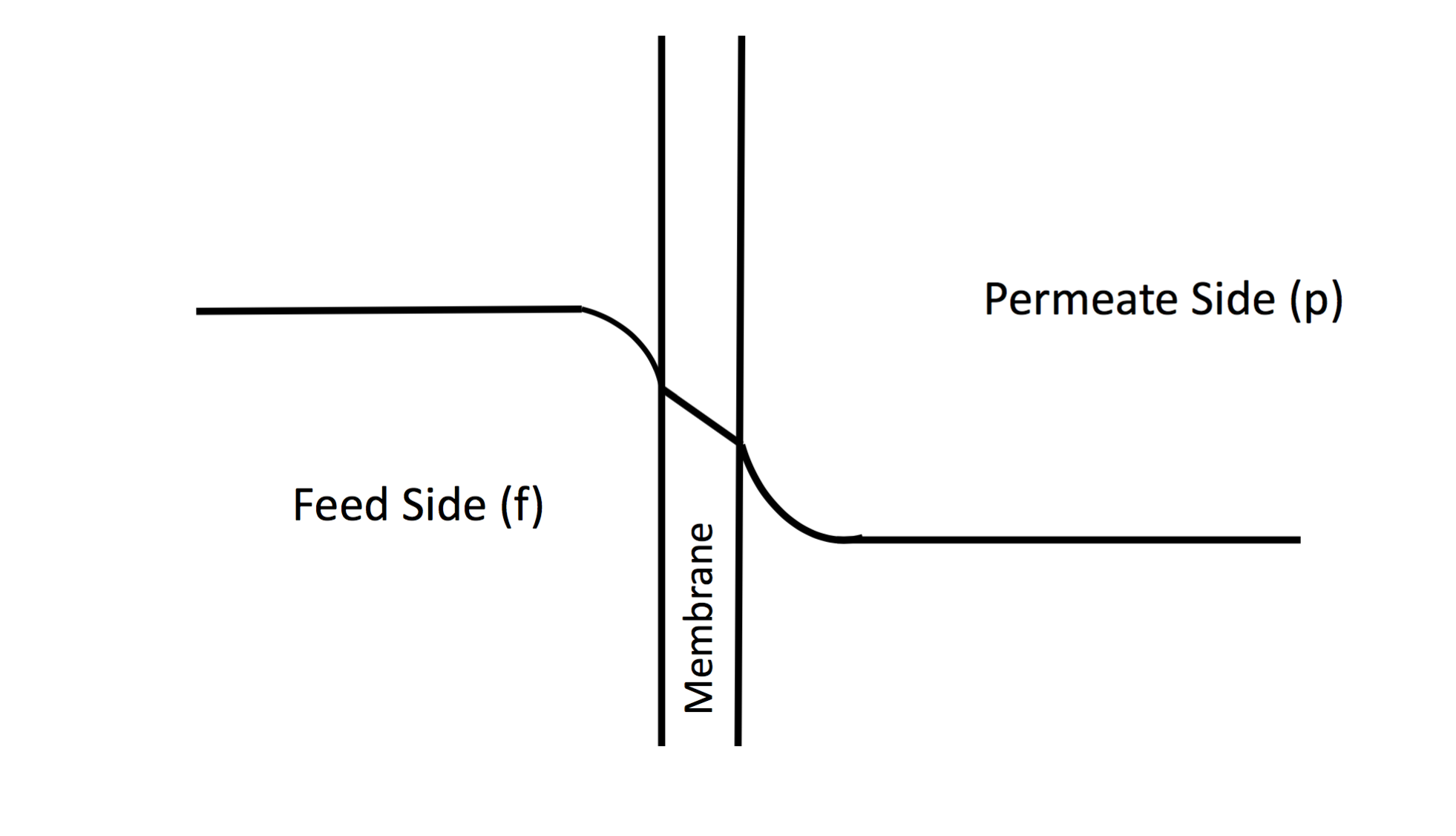}}
\caption{Temperature evolution from feed (f) to permeate (p) in DCMD} \label{flux}
\end{figure}
The mathematical model for the evolution of the temperatures in the device involves a diffusion and a convection terms. The equations write, see \cite{shim2015solar}
 \begin{align}
 \label{edp}
 \p t f(t,x,y) &= \af \Delta f(t,x,y) - \bf\,\p y f(t,x,y)\qquad\text{ for $t\ge0$ and $(x,y)\in \Omega_f$}\\
  \p t p(t,x,y) &= \ap \Delta p(t,x,y) - \bp\,\p y p(t,x,y)\qquad\text{for $t\ge0$ and $(x,y)\in \Omega_p$\,.}
 \end{align}
The coefficients $\af$, $\bf$  and $\ap$  are positive and assumed to be constant: they depend on the thermal conductivity and the densities of the fluids (see~\cite{shim2015solar} ); specifically they are defined as follows
\begin{align*}
\alpha_{f} &=\frac{\kappa_f}{\rho_f\,C_{f}}, & \alpha_{p} & = \frac{\kappa_p}{\rho_p\,C_{p}}\,.
\end{align*}
Here $\kappa_{k},\rho_{k}$ and $C_{k}$, ($k\in\{p,f\}$) denote respectively the thermal conductivity of fluid, liquid density of fluid and specific heat capacity of fluid. The coefficients $\beta_f > 0$ and  $\beta_p$ denote the velocities of the flow in the feed and permeate side respectively.   The velocity  $\bp$ in the permeate is negative in the counter-current case, Fig.~\ref{dessin} (a) and positive for the co-current presentation,  Fig.~\ref{dessin} (b). The boundary  conditions are a mix of Dirichlet, von Neumann and Robin conditions, they are:
 \paragraph{On the boundary $\mathrm{B_1}$}
 \begin{equation}
 \label{B1}
 f(t,x,0)=T_{f} \qquad \text{for every   $0\le x\le \ell$}
 \end{equation}
 
 \paragraph{On the boundary $\mathrm{B_2}$}
 \begin{equation}
 \label{B2}
\p x  f(t,0,y)=0 \qquad \text{for every $0\le y\le L$}
 \end{equation}
 
  \paragraph{On the boundary  $\mathrm{B_3}$}
 \begin{equation}
 \label{B3}
\p y  f(t,x,L)=0 \qquad \text{for every $0\le x\le \ell$}
 \end{equation}

\paragraph{On the boundary $\mathrm{B_4}$}
\begin{equation}\label{B4}
\begin{aligned}
 p(t,x, L)&=T_p\quad \text{for  every $\ell+ \delta_m \le x\le 2\ell + \delta_m$},\,\text{ Fig.~\ref{dessin}\ (a)}\\
 p(t,x, 0)&=T_p\quad \text{for  every $\ell+ \delta_m \le x\le 2\ell + \delta_m$}, \,\text{ Fig.~\ref{dessin}\ (b)}
 \end{aligned}
\end{equation}

\paragraph{On the boundary $\mathrm{B_5}$}
\begin{equation}
\label{B5}
\p x p(t, 2\ell + \delta_m,y)=0\qquad \text{for every $0\le y \le L$}
\end{equation}

  \paragraph{On the boundary $\mathrm{B_6}$}
 \begin{equation} \label{B6}
\begin{aligned}
\p y  p(t,x,0)&=0 \qquad \text{for every $\ell + \delta_m \le x\le 2\ell + \delta_m$}\,\text{ Fig.~\ref{dessin}\ (a)}\\
\p y  p(t,x,L)&=0 \qquad \text{for every $\ell + \delta_m \le x\le 2\ell + \delta_m$}, \,\text{Fig.~\ref{dessin}\ (b)}\\
 \end{aligned}
 \end{equation}
 
\paragraph{On the interfaces $\mathrm{I}_f$ ane $\mathrm{I}_p$}
\begin{align}
\label{If}
k_f \p x f(t,\ell,y)&=-\biggl(J \lambda + \frac{k_m}{ \delta_m}f(t,\ell,y)-\frac{k_m}{ \delta_m}p(t,\ell+ \delta_m,y)\biggr)\\
\label{Ip}
k_p \p x p(t,\ell,y)&=-( J \lambda + \frac{k_m}{ \delta_m}f(t,\ell,y)-\frac{k_m}{ \delta_m}p(t,\ell+ \delta_m,y))
\end{align}
   The surface temperature on the feed side of the membrane equals the feed temperature $f$ and the surface temperature on the permeate side of the membrane equals the bulk temperature $p$ of the condensing fluid. Nevertheless, the process is known to suffer from temperature polarization as depicted in figure \ref{flux} causing a decrease in permeate fluxes~\cite{alsaadi2014experimental}.

The term $J$ denotes the permeate flux through the membrane; the mass transport mechanism in the membrane pores is governed by three basic mechanisms known as: Knudsen diffusion, molecular diffusion and Poiseuille flow, \cite{naidu2017transport, shim2015solar} and~\cite[Chapter 10]{khayet_book}. The term $ \lambda$ is the latent heat of water: these terms depend on the temperature;   the product $J\, \lambda$ is  very small (about $10^{-6}$ ) and is neglected in the sequel. the terms $k_f$, $k_p$ and $k_m$ are  thermal conductivity coefficients. 
 
 As this model does not take into account the physical phenomena inside the  membrane, we shall rewrite it in such a way that the two unknown functions   
 $f$ and $p$ are defined on the same domain. To this end, one make the following change of unknown function: $\tilde p(t,x,y)=p(t,2\ell + \delta_m-x,y)$ where $(x,y)\in[0,\ell]\times [0,L]$.
 Hereafter, the partial differential equations as well as the boundary conditions are rewritten with the unknown functions $f$ and $\tilde p$. For the sake, of readability, we keep the notation $p(t,x,y)$ (instead of $\tilde p$), moreover, without loss of generality, we assume that $\ell=1$. The rest of the paper will examine the counter current case. However, the findings remain true for the co-current process.
The domain of definition of the PDE's is $\Omega :=(0,1)\times(0,L)$, the considered advection diffusion system writes 

\begin{systn}
\label{glv1.1}
& \partial_{t}f(t,x,y)-\alpha_{f}\Delta f(t,x,y)+\beta_{f}\partial_{y}f(t,x,y) =0, &t\ge0, (x,y)\in \Omega\\
& \partial_{t}p(t,x,y)-\alpha_{p}\Delta p(t,x,y)-\beta_{p}\partial_{y}p(t,x,y) =0 &t\ge0, (x,y)\in \Omega\\
&  f(t,x,0)=T_f & t\ge0, 0\le x\le 1,\\
  &\partial_{x} f(t,0,y) =0 & t\ge0, 0\le y \le L,\\
 & \partial_{y} f(t,x,L)  =0& t\ge0, 0\le x \le 1,\\
&   p(t,x,L) =T_{p} & t\ge0, 0\le x \le 1,\\
  & \partial_{x} p(t,0,y) =0 & t\ge0, 0\le y \le L,\\
&  \partial_{y} p(t,x,0) = 0 &t\ge0, 0\le x \le 1,\\
&   \partial_{x} f(t,1,y)  =-\gamma_f \left(f(t,1,y)-p(t,1,y)\right) & t\ge0, 0\le y \le L,\\
& \partial_{x} p(t,1,y) =\gamma_p \left(f(t,1,y)-p(t,1,y)\right) & t\ge0, 0\le y \le L,\\
& f(0,x,y) =f_{0}(x,y) & (x,y)\in \Omega,\\
&   p(0,x,y) =p_{0}(x,y) &(x,y)\in \Omega\,.
 \end{systn}
 Here the constants $ \gamma_f$ and $ \gamma_p$ are respectively equal to $k_m( \delta_mk_f)^{-1}$ and $k_m( \delta_mk_p)^{-1}$. $T_{f}, T_{p},f_{0}$ and $p_{0}$ are the initial data of the system.
 
 \begin{remark}
A simplification of the DCMD model  has been proposed  in~\cite{eleiwi2016dynamic} under appropriate  physical assumptions. Indeed,  the vertical thermal diffusivity for the  considered geometry has been neglected, the width $\ell $ being sufficiently small compared to the length $L$.  This assumption is based on the fact that the horizontal diffusivity is dominant. 
\begin{syst}
& \partial_{t}f(t,x,y)-\alpha_{f}\partial_{xx} f(t,x,y)+\beta_{f}\partial_{y}f(t,x,y) =0, &t\ge0, (x,y)\in \Omega\\
& \partial_{t}p(t,x,y)-\alpha_{p}\partial_{xx} p(t,x,y)-\beta_{p}\partial_{y}p(t,x,y) =0 &t\ge0, (x,y)\in \Omega\\
&  f(t,x,0)=T_f & t\ge0, 0\le x\le 1,\\
  &\partial_{x} f(t,0,y) =0 & t\ge0, 0\le y \le L,\\
&   p(t,x,L) =T_{p} & t\ge0, 0\le x \le 1,\\
  & \partial_{x} p(t,0,y) =0 & t\ge0, 0\le y \le L,\\
&   \partial_{x} f(t,1,y)  =-\gamma_f \left(f(t,1,y)-p(t,1,y)\right) & t\ge0, 0\le y \le L,\\
& \partial_{x} p(t,1,y) =\gamma_p \left(f(t,1,y)-p(t,1,y)\right) & t\ge0, 0\le y \le L,\\
& f(0,x,y) =f_{0}(x,y) & (x,y)\in \Omega,\\
&   p(0,x,y) =p_{0}(x,y) &(x,y)\in \Omega\,.
 \end{syst}
 \end{remark}


\section{Existence and uniqueness of the solution of system~\eqref{glv1.1}}\label{sec-exist-unic}
In order to ensure the existence of a solution to system~\eqref{glv1.1}, appropriate regularity assumptions on the initial datum are required. We next sharpen the regularity of this data.
\subsection{The operator $ {A}$ with inhomogeneous boundary conditions}
In this section, we consider the following systems of partial differential equations
\begin{systn}
\label{op-inhom}
& \alpha_{f}\Delta f(x,y)-\beta_{f}\partial_{y}f(x,y) =0, &(x,y)\in \Omega\\
&-\alpha_{p}\Delta p(x,y)+\beta_{p}\partial_{y}p(x,y) =0 & (x,y)\in \Omega\\
&  f(x,0)=T_f & 0\le x\le 1,\\
  &\partial_{x} f(0,y) =0 &  0\le y \le L,\\
 & \partial_{y} f(,x,L)  =0& 0\le x \le 1,\\
&   p(x,L) =T_{p} &  0\le x \le 1,\\
  & \partial_{x} p(0,y) =0 &  0\le y \le L,\\
&  \partial_{y} p(x,0) = 0 & 0\le x \le 1,\\
&   \partial_{x} f(1,y)  =-\gamma_f \bigl(f(1,y)-p(1,y)\bigr) &  0\le y \le L,\\
& \partial_{x} p(1,y) =\gamma_p \bigl(f(1,y)-p(1,y)\bigr) &  0\le y \le L\,.
 \end{systn}
If $T_f=T_p=0$, this system writes $ {A}(f,p)=0$ with $ {A}$ the operator defined above; in this case the unique solution of the 
system is $(f,p)\equiv(0,0)$ because   the m-dissipative operator $ {A}:\mathcal{D}( {A})\to {\bigl[L^2( \Omega)\bigr]}^2$ is 
invertible. In the general case, we shall use some tools from the theory of boundary control systems 
(see \emph{e.g.} \cite[Chap. 10]{TucWeiss}). In system~\eqref{op-inhom}, we shall regard $T_f$ and $T_p$ as boundary controls, 
and we introduce the following spaces and operators:
\begin{itemize}
\item
the solution space  $\mathbf{Z}$ is defined as those pairs $(f,p)\in H^2( \Omega)\times H^2(\Omega)$ satisfying the following homogeneous boundary conditions:
\begin{itemize}
\item
for every $0\le y\le L$, $\partial_{x} f(0,y) =\partial_{x} p(0,y)=0$\,;
\item
for every $0\le x\le1$, $\partial_{y} f(x,L)  = \partial_{y} p(x,0) = 0$\,;
\item
for every $0\le y\le L$, $\partial_{x} f(1,y)  =-\gamma_f \left(f(1,y)-p(1,y)\right)$\,, 
and $\partial_{x} p(1,y) =\gamma_p \left(f(1,y)-p(1,y)\right) $\,.
\end{itemize}
\item
the state space $\mathbf{X}$ is the space $L^2( \Omega)\times L^2( \Omega)$\,;
\item
the input space $\mathbf{U}$ is the space $L^2(0,1)\times L^2(0,1)$\,.
\end{itemize}
Notice that $\mathbf{Z}\subset \mathbf{X}$ with continuous embedding. We consider the operator $ {L}:\mathbf{Z}\to \mathbf{X}$ defined as
\begin{equation*}
 {L}(f,p) = \bigl(\alpha_f \Delta f- \beta_f\p yf, \alpha_p \Delta p+ \beta_p \p yp\bigr)
\end{equation*}
and the operator $ {G}: \mathbf{Z}\to \mathbf{U}$ defined as
\begin{equation*}
 {G}(f,p)= \bigl(f(\cdot,0),p(\cdot,L)\bigr)\,.
\end{equation*}
Operator $ {L}$ is obviously bounded, this is true also for operator $ {G}$
\begin{lemma}
The linear operator $G: \mathbf{Z}\to\mathbf{U}$ is bounded.
\end{lemma}
\begin{proof}
Consider the function $F$ defined for $x\in[0,1]$ as
\[
F(x)=\int_0^Lf^2(x, y)\d y,
\]
 we have
\begin{align*}
F(0)&=F(x)-  \int_0^x \diff{F(\xi)}{ \xi}\,\d \xi\\
& = F(x)-  2\int_0^x \left(\int_0^Lf(\xi, y)\p x f(\xi, y)\,\d \xi\right)\d y\\
&\le F(x) +    2\int_0^1 \left(\int_0^L|f(\xi, y)|\,|\p x f(\xi, y)|\,\d \xi\right)\d y\\
& \le F(x) +     \int_0^1 \left(\int_0^Lf^2(\xi, y)\,\d \xi\right)\d y +  \int_0^1 \left(\int_0^L(\p x f)^2(\xi, y)\,\d \xi\right)\d y\,.
\end{align*}
By integrating this inequality with respect to $x$ on the interval $[0,1]$, we obtain
\begin{align*}
F(0) & \le \int_0^1F(x)\,\d x + \int_ \Omega f^2\d x\d y + \int_ \Omega(\p y f)^2\d x\d y
\end{align*}
which reads
\begin{align*}
\|f(\cdot,0)\|_{L^2(0,1)} &\le 2\|f\|^2_{L^2( \Omega)} + \|\p x f\|^2_{L^2( \Omega)}
\end{align*}
and so we have
\begin{equation*}
\|f(\cdot,0)\|_{L^2(0,1)} \le 2 \|f\|_{H^2( \Omega)}\,.
\end{equation*}
Clearly,    the same inequality is true for $p$, which proves that $ {G}$ is a bounded operator. 
\end{proof}

Now we have.
\begin{proposition}
\label{prop:boundary-cont}
The operators $G$and $L$ satisfy the following properties.
\begin{enumerate}
\item
$ {G}$ is onto;
\item
$\Ker  {G}$ is dense in $\mathbf{X}$\,;
\item
$- {L}$ restricted to $\Ker  {G}$ is onto;
\item
$\Ker (- {L})\cap\Ker  {G}=\{0\}$\,.
\end{enumerate}
\end{proposition}

The two first points are obvious. As regards   the third point, notice first that $\Ker  {G}\subset\mathcal{D}( {A})$ moreover,   given $(u,v)\in \mathbf{X}$, we 
know that there exist $(f,p)\in \mathcal{D}( {A})$ such that $\mathcal{A}(f,p)=(u,v)$.  In order to prove that the pair 
$(f,g)$ is in $\Ker \mathbf{G}$,   we have to  show that  $f$ and $p$ are in $H^2( \Omega)$.  The proof of the regularity of the weak solution of an elliptic equation is a classical result (see \emph{e.g.}~\cite{evans2010partial} or~\cite{Brezis}) but this result assumes Dirichlet or Neumann boundary conditions.
 In~\cite{Faierman}, M.~Faierman proves the regularity of the weak solution of an elliptic equation $Mq=r$ where $M$ is an elliptic operator defined on a rectangle $R$  and where, as for the system considered in this paper, the boundary conditions are of mixed type: Dirichlet, Neumann of Robin. More specifically, under the condition that   $r\in L^2(R)$, Faierman proves that $q$ is in $H^2(R)$.  This proof is intended for an elliptic equation whose unknown function $f$ takes its values in $\R$ while, in this paper,  our unknown is a couple of functions $(f,p)$\,; nevertheless, it suffices to adapt slightly the reasoning of Faierman to prove that 
 $f$ and $p$ are in $H^2( \Omega)$. 
 
 \begin{proposition}
 If both functions $u$ and $v$ are in $L^2( \Omega)$, the unique pair $(f,p)\in\mathcal{D}(A)$ such that $A(f,p)=(u,v)$ belongs to 
 $\mathcal{H}^2( \Omega)\times\mathcal{H}^2( \Omega)$\,.
\end{proposition}
\begin{proof}
 Hereafter, we treat only the case of $f$, the reasoning for $p$ being the same than for $f$.  First, we define two extensions of $f$: $f_1$ on $ \Omega_1=[-1,1]\times[0,L]$ and $f_2$ on $ \Omega_2=[0,2]\times[0,L]$ as follows
\begin{align*}
f_1(x,y) &=  \begin{cases}
\phi_{1}(x)f(x,y) & \text{if \ensuremath{0\le x\le1}} \\
\phi_{1}(-x)f(-x,y) & \text{if \ensuremath{-1\le x \le0} }
\end{cases}\\[0.75em]
f_2(x,y)&=
 \begin{cases}
\bigl(1-\phi_{1}(x)\bigr)f(x,y) & \text{if \ensuremath{0\le x\le1}}\protect\\
-\bigl(1-\phi_{1}(2-x)\bigr)f(2-x, y) +2 f(1,y) & \text{if \ensuremath{1\le x \le2} }
 \end{cases}
\end{align*}
where $ \phi_1$ is a $C^\infty$ function such that 
\[
\phi_{1}(x)=
\begin{cases}
1 & \text{if \ensuremath{x \le1/4}}\\
0 & \text{si \ensuremath{3/4\le x\le1}}
\end{cases}
\]
and $0\le \phi_1(x)\le1$ for all $x\in\R$. 
Notice that $f_1$ is defined as in~\cite{Faierman} but the definition for $f_2$ (as well as the notations) differs slightly from the one adopted in this paper. First, it is easily shown that $f_i\in H^2( \Omega_i)$ ($i=1,2$), then we shall   show that $f_1$ and $f_2$ can be regarded as weak solutions to some PDE's. We begin with function $f_2$: take  
$\psi_2\in H^1( \Omega_2)$ such that $\psi_2(x,0)=0$ for $0\le x\le2$, first notice that we have
 \begin{multline}
 \label{f2-f}
\alpha_f (\nabla f_2)(\nabla \psi_2) + \beta_f(\p yf_2)\psi_2 = 
\alpha_f(\nabla f)\bigl(\nabla ( \phi_2\psi_2)\bigr) + \beta_f(\p y f)( \phi_2 \psi_2)\\
 - \alpha_f(\p x \phi_2)(\p x f)\psi_2 + \alpha_f(\p x \phi_2)f(\p x \psi_2)
 \end{multline}
 where $ \phi_2(x) := 1- \phi_1(x)$.
  Integrating by parts and taking into account that $\mathcal{A}(f,p)=(u,v)$, we obtain the following equality 
  \begin{equation}\label{f-weak-form}
  \begin{aligned}
   \int_ {\Omega}\alpha_f (\nabla f)\cdot(\nabla\psi_2)\,\d x\d y &+ \int_{ \Omega} \beta_f (\p yf)\psi_2\,\d x\d y \\&=-\int_{ \Omega}  u\,\psi_2\,\d x\d y + \alpha_f\int_0^L \p xf(1,y)\psi_2(1,y)\,\d y
  \end{aligned}
  \end{equation}
  from this equality and~\eqref{f2-f}, and taking into account that $\phi_2(1)=1$, and $ (\p x\phi_2)(0)= (\p x \phi_2)(1)= 0$, we obtain
 \begin{align*}
 \int_ {\Omega}\alpha_f (\nabla f_2)&\cdot(\nabla\psi_2)\,\d x\d y + \int_{ \Omega} \beta_f (\p yf_2)\psi_2\,\d x\d y\\
 &=
 -\int_{ \Omega} ( \phi_2 u+ \alpha_f(\p x \phi_2)(\p x f))\psi_2\,\d x\d y \\
 &\mathrel{\phantom{=}} \mbox{} + 
\alpha_f\int_0^L \p xf(1,y)\psi_2(1,y)\,\d y+  \int_ \Omega \p x \phi_2f\,(\p x\psi_2)\,\d x\d y\\
 & = -\int_{ \Omega} ( \phi_2 u+ \alpha_f(\p x \phi_2)(\p x f))\psi_2\,\d x\d y +
  \alpha_f\int_0^L \p xf(1,y)\psi_2(1,y)\,\d y \\
&\phantom{\mathrel{=}} \mbox{}  + \alpha_f \int_0^L \bigl[(\p x \phi_2)f\psi_2\bigr]_{x=0}^{x=1}\d y 
- \alpha_f\int_ \Omega\bigl( (\p {xx} \phi_2) f + (\p x \phi_2)(\p x f)\bigr)\psi_2\d x\d y\\
& = -\int_{ \Omega} ( \phi_2 u+g_2)\psi_2\,\d x\d y +  \alpha_f\int_0^L (\p x f)(1,y))\psi_2(1,y)\,\d y
 \end{align*}
 where $g_2$ is the function defined as
\begin{equation*}
g_2(x,y) := \alpha_f\bigl(\p{xx} \phi_2(x)\,f(x,y)+2\p x \phi_2(x)\p x f(x,y)\bigr)\,.
\end{equation*}
From this formula, we deduce that,
\begin{multline}
\label{eq-f2}
\int_ {\Omega_2}\alpha_f (\nabla f_2)\cdot(\nabla\psi_2)\,\d x\d y + \int_{ \Omega_2} \beta_f (\p yf_2)\psi_2\,\d x\d y=
-\int_{ \Omega_2} ( \phi_2 u+g_2)^*\psi_2\,\d x\d y \\
+ 2 \alpha_f \int_{ \Omega_2}(\p y f)(1,y)(\p y\psi_2)(x,y)\,\d x\,\d y
\end{multline}
where
\begin{equation*}
(\phi_2 u+g_2)^*:=  \begin{cases}
(\phi_2 u+g_2)(x,y)& \text{if $(x,y)\in \Omega$}\\
-(\phi_2 u+g_2)(2-x,y) & \text{if $(x,y)\in[1,2]\times[0,L]$.}
\end{cases}
\end{equation*}
Due to the second integral in the right-hand member in~\eqref{eq-f2}, function $f_2$ cannot be regarded as the weak solution of a PDE, nevertheless, we can apply the method of difference quotients. In the proof of~\cite[Th. 1, p. 329]{evans2010partial}, an open set $V\subset\bar V\subset \Omega_3$ is fixed and a smooth cutoff function $ \theta$ is chosen ($ \theta$ is equal to 1 on $V$ and to 0 outside an open set $W$ such that 
$U\subset W\subset \bar W\subset \Omega_3$); then    function $\psi_2$ in equality~\eqref{eq-f2} is taken to be equal to 
\begin{equation*}
\psi_2(x,y) = \frac1{h^2}\left( \theta^2(x-h,y)(f(x,y)-f(x-h,y)- \theta^2(x,y)(f(x+h,y)-f(x,y)\right)\,.
\end{equation*}
With this choice of $\psi_2$, the second integral in the right-hand member of~\eqref{eq-f2} is zero
and it follows that we can argue as in the proof of~\cite[Th. 1, p. 329]{evans2010partial}.

To prove the boundary regularity, we can still proceed as in~\cite{evans2010partial}, in this case also, we do not have to take care of the second integral in the right-hand member of~\eqref{eq-f2}. To be more precise, consider a point $(x_0,0)$ of the edge $[0,2]\times\{0\}$ of $ \Omega_2 $ with $1/4\le x_0\le 7/4$. Denote by $U_r$ the half ball $U_r=B(x_0,r)\cap\R^2_+$ where, as usual, $B(x_0,r)$ denotes the open ball of radius $r$ centered at $x_0$ and $\R_+^2=\ens{(x,y)\in\R^2}{y>0}$; moreover $r$ is chosen small enough in order that $B(x_0,2r)$ does not intersect the edge $[0,2]\times\{L\}$. Select a smooth cutoff function $ \sigma$ satisfying
 \begin{syst}
& \sigma\equiv 1 \text{ on $B(x_0,r)$}, &\sigma\equiv0 \text{ on $\R^2\smallsetminus B(x_0,2r)$}\\
& 0\le \sigma\le1\,.
 \end{syst}
 Let $h>0$ be small and write
 \begin{equation}
\label{psi3}
 \psi_2(x,y) := -D_1^{-h}( \sigma^2D_1^h K^a)
 \end{equation}
 where, for any function $K$,  $D_1^h K$ denotes the difference quotient 
 \[
 D_1^hK(x,y):=\frac{K(x+h)-K(x)}h\,.
 \]
With this choice of $\psi_2$, it is easily seen that  the second integral in the right-hand member of~\eqref{eq-f2} vanishes. Thus we   can argue   exactly as in the proof 
of~\cite[Th.4, p. 336]{evans2010partial} in order to establish the following estimate
\[
\int_U\|D_1^h\nabla f_2\|^2\d x\d y\le C\,,
\]
which proves the result. We treat the regularity near the piece of boundary $[1,2]\times\{L\}$ in the same way and we notice that we
do not have to worry about the corner since $f_2$ is zero in some
neighborhoods of the edges $\{1\}\times[0,L]$ and $\{2\}\times[0,L]$.

 The case of $f_1$ is slightly simpler, take $\psi_1$ in $H^1( \Omega_1)$ and  such that $\psi_1(x,0)=0$ for $-1\le x\le 1$, we have
where $g_3$ is defined similarly as $g_1$.  We obtain for $f_1$ a formula analogous to~\eqref{f2-f}, from this formula and~\eqref{f-weak-form}, we get
\begin{equation*}
 \int_ {\Omega}\alpha_f (\nabla f_1)\cdot(\nabla\psi_1)\,\d x\d y + \int_{ \Omega} \beta_f (\p yf_1)\psi_1\,\d x\d y = 
  -\int_{ \Omega} ( \phi_1 u+g_1)\psi_1\,\d x\d y\,;
\end{equation*}
notice that in this case, as $\phi_1(1)=0$, we do not have to deal with a term like the second integral in the right-hand member of~\eqref{eq-f2}. From this equality, we deduce 
\begin{equation*}
 \int_ {\Omega_1}\alpha_f (\nabla f_1)\cdot(\nabla\psi_1)\,\d x\d y + \int_{ \Omega_1} \beta_f (\p yf_1)\psi_1\,\d x\d y = 
  -\int_{ \Omega_1} ( \phi_1 u+g_1)^*\psi_1\,\d x\d y\,;
\end{equation*}
here $g_1$ and $(\phi_1u+g1)^*$ are defined analogously as $g_2$ and $(\phi_2u+g_2)^*$. 
These computations show that $f_1$ is a weak solution of the following problem: 
\begin{gather*}
\alpha_f \Delta f_1- \beta_f\p y f_1 =(u \phi_1+g_1)^* \text{ on $ \Omega_1$}\\
f_1=0 \text{ on $\partial \Omega_1\smallsetminus \Gamma_1$}\\
\diff{f_1}{\nu}=0 \text{ on $ \Gamma_1$}
\end{gather*}
here, as usual, $\mathrm{d}f_i/\mathrm{d}\nu$ denotes the normal derivative and $ \Gamma_1$ is the edge of the rectangle $ \Omega_1$ defined as $  \Gamma_1=[-1,1]\times\{L\}$. Classical results 
(see \emph{e.g.}~\cite{evans2010partial})  allow us to assert that $f_1$    is in $H^2_\mathrm{loc}( \Omega_i)$. Concerning the regularity up to the boundary, we can argue exactly as in  in~\cite[Th. 4, p. 336]{evans2010partial}; thus, function $f_1$  is in $H^2( \Omega_1)$. Consider now the function $f_1+f_2$ restricted to $ \Omega$, this function is in $H^2( \Omega)$ and is equal to $f$, thus we proved that $f\in H^2( \Omega)$.
\end{proof}
  
   Take now $(f,p)$ in $\Ker( - {L})\cap\Ker  {G}$, $(f,p)$ belongs to $\mathcal{D}( {A})$, so $ {L}(f,p)=(0,0)$ implies $ {A}(f,p)=(0,0)$ which in turn implies $(f,p)=(0,0)$ because $ {A}$ is injective.  This achieve the proof of Theorem~\ref{prop:boundary-cont}.
 
 Thus, we can apply \cite[ Proposition~10.1.2]{TucWeiss}: there exist a unique operator $ {B}\in\mathcal{L}(\mathbf{U},\mathbf{X}_{-1})$ such that 
\begin{equation*}
 {L}= \mathcal{A}+ {B} {G}\,.
\end{equation*}
Here $\mathcal{A}$ denotes the restriction of $ {L}$ to $\Ker  {G}$ and is thus identical to the operator $ {A}$ defined in section~\ref{deft-op-A}; $\mathbf{X}_{-1}$ denotes the completion of the space $\mathbf{X}$ with respect to the norm $\|(u,v)\|_{-1}=\|\mathcal{A}^{-1}(u,v)\|$\,.

Operator $ {A}$ is m-dissipative, therefore,  it is the generator of a contraction semigroup. Moreover from equalities~\eqref{dissp-A0} and~\eqref{dissip-B0},   we obtain 
\begin{align*}
\Ps{ {A}(f,p)}{(f,p)}_{{[L^2( \Omega)]}^2} & \le - \left(\|\nabla f\|^2+\|\nabla p\|^2\right)
\end{align*}
from the Poincar\'e's inequality,we have 
\begin{align*}
\int_ \Omega \bigl(\|\nabla\|^2+\|\nabla p\|^2\bigr) \, \d x \d y \ge C\int_ \Omega (f^2+p^2) \, \d x \d y\,,
\end{align*}
which implies that $ {A}\le -C\Id$\,. Denoting by $\T_t$  the semigroup generated by $ {A}$, we thus have 
$\|\T_t\|\le e^{-Ct}$\,. As $ {A}$ is the generator of a strongly continuous semigroup, for every $T>0$, $(f_0,p_0)\in \mathbf{Z}$, and 
$(u,v)\in \mathbf{U}$ such that $ {G}(f_0,p_0)=(u,v)$, the equation
\begin{align*}
\diff{(f,p)}{t}&= {L}(f,p)=\mathcal{A}(f,p)+ {B} {G}(u,v),\\
(f(0),p(0)) &= (f_0,p_0),
\end{align*}
admits a unique solution $(f,p)$ such that $(f,p)\in C([0,T];\mathbf{Z})\cap C^1([0,T];\mathbf{X})$. Moreover as $\T_t$ is exponentially stable, we have $\lim_{t\to\infty}(f(t),p(t)) = (f_\infty,p_\infty)$  where $(f_\infty,p_\infty)$ is the unique solution of the equation
$ {B} {G}(f_0,p_0)=- {A}(f_\infty,p_\infty)$. Thus, we have proved that, given any pair of initial conditions $(f_0,p_0)$ and any  pair of input temperatures $(T_f,T_p)$, there exists a unique solution to system~\eqref{glv1.1} that tends exponentially towards an asymptotic state $(f_\infty,p_\infty)$, as $t$ tends to infinity.

\section{Co-current operator}\label{cocurentsection}

In this section, we consider a  DCMD model with co-current; the equations modeling this device are the same 
as~\eqref{glv1.1} except for the sign of $ \beta_p$\,: they write
\begin{systn}
\label{glv1.2}
& \partial_{t}f(t,x,y)-\alpha_{f}\Delta f(t,x,y)+\beta_{f}\partial_{y}f(t,x,y) =0, &t\ge0, (x,y)\in \Omega\\
& \partial_{t}p(t,x,y)-\alpha_{p}\Delta p(t,x,y)+\beta_{p}\partial_{y}p(t,x,y) =0 &t\ge0, (x,y)\in \Omega\\
&  f(t,x,0)=T_f & t\ge0, 0\le x\le 1,\\
  &\partial_{x} f(t,0,y) =0 & t\ge0, 0\le y \le L,\\
 & \partial_{y} f(t,x,L)  =0& t\ge0, 0\le x \le 1,\\
&   p(t,x,0) =T_{p} & t\ge0, 0\le x \le 1,\\
  & \partial_{x} p(t,0,y) =0 & t\ge0, 0\le y \le L,\\
&  \partial_{y} p(t,x,L) = 0 &t\ge0, 0\le x \le 1,\\
&   \partial_{x} f(t,1,y)  =-\gamma_f \left(f(t,1,y)-p(t,1,y)\right) & t\ge0, 0\le y \le L,\\
& \partial_{x} p(t,1,y) =\gamma_p \left(f(t,1,y)-p(t,1,y)\right) & t\ge0, 0\le y \le L,\\
& f(0,x,y) =f_{0}(x,y) & (x,y)\in \Omega,\\
&   p(0,x,y) =p_{0}(x,y) &(x,y)\in \Omega\,.
 \end{systn}
Introducing the following change of variables
\begin{align}
\label{chgt-var}
g(t,x,y) & =f(t,x,y) e^{-\frac{\beta_{f}}{2\alpha_{f}}y} ,& q(t,x,y)=p(t,x,y) e^{-\frac{\beta_{p}}{2\alpha_{p}}y}\,.
\end{align}
 we then have 
\begin{align*}
\p{xx}g(t,x,y)&=\p{xx}f(t,x,y)\exp\biggl(-\frac{\beta_{f}}{2\alpha_{f}}y\biggr)\,, \\[0.5ex]
\p yg(t,x,y)&=\Bigl(\p yf(t,x,y) - \frac{\beta_{f}}{2\alpha_{f}}f(x,y)\Bigr)\exp\biggl(-\frac{\beta_{f}}{2\alpha_{f}}y\biggr)\,,\\[0.75ex]
\p{yy}g(t,x,y)&=\Bigl(\p{yy}f(t,x,y) -\frac{\beta_{f}}{\alpha_{f}}\p yf(t,x,y)\Bigr)\exp\biggl(-\frac{\beta_{f}}{2\alpha_{f}}y\biggr)
+\frac{\beta^{2}_{f}}{4\alpha^{2}_{f}}g(t,x,y)\,,
\end{align*}
and so
\begin{align*}
\Bigl(\alpha_{f}\Delta f(t,x,y)- \beta_f\p yf(t,x,y)\Bigr)\exp\biggl(-\frac{\beta_{f}}{2\alpha_{f}}y\biggr) 
&= \alpha_f  \Delta g(t,x,y) - \frac{ \beta_f^2}{4 \alpha_f}g(t,x,y)\,.
\end{align*}
Similar computations lead to
\begin{align*}
\Bigl(\alpha_p\Delta p(t,x,y)- \beta_p\p yp(t,x,y)\Bigr)\exp\biggl(-\frac{\beta_{p}}{2\alpha_{p}}y\biggr) 
&= \alpha_p  \Delta q(t,x,y) - \frac{ \beta_p^2}{4 \alpha_p}q(t,x,y)\,.
\end{align*}

Regarding the  boundary conditions, we have
\begin{align*}
g(t,x,0)&=T_f\\
\p xg(t,0,y)&=0\\
\p yg(t,x,L) & =-\frac{\beta_{f}}{2\alpha_{f}}g(t,x,L)\\
q(t,x,0) & =T_{p}\\
\p xq(t,0,y) & =0\\
\p yq(t,x,L)& = -\frac{\beta_{p}}{2\alpha_{p}}q(t,x,L)\\
\p xg(t,1,y)&=-\gamma_f \biggl(g(t,1,y)-q(t,1,y)e^{\left(-\frac{\beta_{f}}{2\alpha_{f}}+\frac{\beta_{p}}{2\alpha_{p}}\right)y}\biggr)\\
\p xq(t,1,y)&=\p xp(t,1,y)e^{-\frac{\beta_{p}}{2\alpha_{p}}y}=\gamma_p \biggl(e^{\left(\frac{\beta_{f}}{2\alpha_{f}}-\frac{\beta_{p}}{2\alpha_{p}}\right)y}g(1,y)-q(1,y)\biggr)\,.
\end{align*}
We assume that the flow velocities can be adjusted in such a way that
\begin{equation}\label{Assumptiondiag}
\frac{\beta_{f}}{2\alpha_{f}}=\frac{\beta_{p}}{2\alpha_{p}}\,.
\end{equation}
Under the change of unknown functions~\eqref{chgt-var}, and with the assumption~\eqref{Assumptiondiag}, system~\eqref{glv1.2} becomes
\begin{systn}
\label{glv1.3}
& \partial_{t}g(t,x,y)-\alpha_{f}\Delta g(t,x,y)+\frac{\beta_{f}^2}{4 \alpha_f} g(t,x,y) =0, &t\ge0, (x,y)\in \Omega\\
& \partial_{t}q(t,x,y)-\alpha_{p}\Delta q(t,x,y)+\frac{\beta_{p}^2}{4 \alpha_p}q(t,x,y) =0 &t\ge0, (x,y)\in \Omega\\
&  g(t,x,0)=T_f& t\ge0, 0\le x\le 1,\\
  &\p xg(t,0,y)=0 & t\ge0, 0\le y \le L,\\
 &\p yg(t,x,L)  =-\frac{\beta_{f}}{2\alpha_{f}}g(t,x,L)& t\ge0, 0\le x \le 1,\\
&   q(t,x,0)  =T_{p} & t\ge0, 0\le x \le 1,\\
  & \partial_{x} q(t,0,y) =0 & t\ge0, 0\le y \le L,\\
&\p yq(t,x,0) = -\frac{\beta_{p}}{2\alpha_{p}}q(t,x,0)&t\ge0, 0\le x \le 1,\\
&   \partial_{x} g(t,1,y)  =-\gamma_f \left(g(t,1,y)-q(t,1,y)\right) & t\ge0, 0\le y \le L,\\
& \partial_{x} q(t,1,y) =\gamma_p \left(g(t,1,y)-q(t,1,y)\right) & t\ge0, 0\le y \le L,\\
& g(0,x,y) =g_{0}(x,y):= f_0(x,y)e^{\bigl(-\frac{ \beta_f}{2 \alpha_f}y\bigr)} & (x,y)\in \Omega,\\
&   q(0,x,y) =q_{0}(x,y):= p_0(x,y)e^{\bigl(-\frac{ \beta_p}{2 \alpha_p}y\bigr)}  &(x,y)\in \Omega\,.
 \end{systn}

We shall prove no that the operator, denoted by $\tilde{{A}}$, and  related to this system is diagonalizable. The spaces related to this operator will be slightly different from the ones related to operator $ {A}$. First, we introduce the space  
$\mathbf{\tilde{E}}$ defined as  the set of those pairs
 $(g,q)$  in  $\bigl[H^{1}(\Omega) \cap C^{1}(\bar\Omega)\bigr]^{2}$ such that
 \begin{itemize}
 \item
 $ g(x,0)=q(x,0) = 0$  for every $x\in(0,1)$\,;
 \item
 $\partial_{x} g(0,y)=\partial_{x} q(0,y) = 0$,   for every $ y \in (0,L)$\,;
 \item
$\p yg(x,L)  =-\frac{\beta_{f}}{2\alpha_{f}}g(x,L)$for every $x\in(0,1)$\,;
\item 
   $\p yq(x,L) = -\frac{\beta_{p}}{2\alpha_{p}}q(x,L)$ for every $x\in(0,1)$\,;
\item
$ \partial_{x}g(1,y)= - \gamma_f\bigl(g(1,y)-q(1,y)\bigr)$, for every $y\in(0,L)$\,;
\item
$ \partial_{x}q(1,y)=\gamma_p\bigl(g(1,y)-q(1,y)\bigr)$ for every $y\in(0,L)$\,.
 \end{itemize}

As section~\ref{deft-op-A}, the space $H^1( \Omega)\times H^1( \Omega)$ is equipped with the product topology, making it an Hilbert space whose norm is defined by~\eqref{def-norme-H1}. On $L^2( \Omega)\times L^2( \Omega)$, we consider again the inner product given by~\eqref{ps-L2}.  We denote then by $\tHbc$ the closure of $\mathbf{\tilde E}$ in 
$\bigl[H^{1}(\Omega)\bigr]^{2}$\,; recall that the induced norm on $\tHbc$ is also defined 
by~\eqref{def-norm-Hbc}.

The  operator $\tilde{ {A}}$ is then defined as follows: its domain is given by 
\[
\mathcal{D}(\tilde{ {A}}):=
\ens{(g,q)\in \tHbc}{(\Delta g,\Delta q)\in \bigl[L^{2}(\Omega)\bigr]^{2} }\,;
\]
and, for every $(g,q)\in\mathcal{D}(\tilde{ {A}})$, 
\[
\tilde{ {A}}(g ,q ) = \bigl(\alpha_{f}\Delta g - \frac{ \beta_f^2}{4 \alpha_f}g , \alpha_{p}\Delta q -\frac{ \beta_p^2}{4 \alpha_p}\bigr)\,. 
\]
Using similar arguments 	as in section~\ref{sec:op-A0}, we can prove that $\tilde{ {A}}$ is m-dissipative and diagonalizable;  moreover, reasoning as in section~\ref{sec-exist-unic}, we can prove that, given an initial condition in $\mathcal{D}(\tilde{ {A}})$, system~\eqref{glv1.3} has a unique solution and that this solution converges, as $t\to\infty$, towards the solution $(g_\infty,q_\infty)$ of the equation $\tilde{ {B}}\tilde{ {G}}(g_0,q_0)=-\tilde{ {A}}(g_\infty,q_\infty)$.
\section{Output Tracking for 2D Direct Contact Membrane Distillation System}\label{section6}
 In this section, we discuss  the  output tracking for the DCMD system. We  propose an output-feedback strategy for the output-tracking of the DCMD, in presence of  unknown disturbances. We use the active disturbance rejection control (ADRC) method. We start our analysis by the co-current configuration; the counter-current configuration is briefly discussed in the conclusion. The system \eqref{glv1.2} can be written in the matrix form 
\begin{equation}\label{BBGL0.03}
\begin{cases}
\displaystyle \partial_{t}w(t,x)-\alpha\Delta w(t,x)+B.\nabla w(t,x)=0&\ t>0,\,\,\, x\in\Omega,\\
\displaystyle \frac{\partial w(t,x)}{\partial \nu}=\mathcal{M}w(t,x)&\ t>0,\,\,\, x\in\Gamma_{4},\\
\displaystyle \frac{\partial w(t,x)}{\partial \nu}=u(t)&\ t>0,\,\,\, x\in \Gamma_{3},\\
\displaystyle \frac{\partial w(t,x)}{\partial \nu}=0& \ t>0,\,\,\, x\in \Gamma_{2},\\
\displaystyle \frac{\partial w(t,x)}{\partial \nu}=d(t)&\ t>0,\,\,\, x\in \Gamma_{1},\\
 \displaystyle w(0,x)=w_{0}(x)&  x\in\Omega,\\
 \displaystyle w(t,x)=y_{m}(t,x)& \ t>0,\,\,\, x\in \Gamma_{1},\\
 \displaystyle y_{r}(t,x)=w(t,x)& \ t>0,\,\,\, x\in \Gamma_{3},
\end{cases}
\end{equation}
where $\Gamma_{1}=(0,1)\times\{0\}$, $\Gamma_{2}=\{0\}\times(0,L)$, $\Gamma_{3}=(0,1)\times\{L\}$ and $\Gamma_{4}=\{1\}\times(0,L)$, $w(t,x)=\begin{bmatrix}
    f(t,x), p(t,x)
\end{bmatrix}^{T}$, with the input $u(t)=\begin{bmatrix}
    u_{1}(t,x), 0
\end{bmatrix}^{T}$, as well as the unknown disturbance $d(t)$. $B=\begin{bmatrix}
    0&\beta_{f}\\ 0&\beta_{p}\\
\end{bmatrix}$, $\alpha=\begin{bmatrix}
   \alpha_{f}&0\\ 0&\alpha_{p}
\end{bmatrix}$, $\mathcal{M}=\begin{bmatrix}
   \gamma_{f}&-\gamma_{f}\\ -\gamma_{p}&\gamma_{p}
\end{bmatrix},$
and  $y_{m}$ is the measured output, $y_{r}$ the performance output signal to be regulated. The system \eqref{BBGL0.03} is discussed on  the boundary space $\mathcal{H}_{\Gamma_{i}}=[L^2(\Gamma_{i})]^{2}$, for all $i\in \{1,2,3,4\}$. We denote by  $\mathcal{H}_{s}^{m}=[H^{m}(\Omega)]^{2}$ the Sobolev space of order $m$ and $\mathcal{H}_{s\Gamma_{i}}^{m}=[H^{m}(\Gamma_{i})]^{2}$ the Sobolev space of order $m$ on the boundary for all $i\in \{1,2,3,4\}$.

 Our goal is to design an output feedback control for, uncertain system \eqref{BBGL0.03} to achieve output tracking $\|e(t)\|_{\mathcal{H}_{\Gamma_{3}}}=\|y_{r}(t)-r(t)\|_{\mathcal{H}_{\Gamma_{3}}}  \longrightarrow 0\quad,\text{as} \,\, t\to\infty$, for a reference signal  $r\in  W^{1,\infty}(0,\infty,\mathcal{H}_{\Gamma_{3}})$, and to reject the external disturbance.
We  propose a one-side  feedback law  to track a desired outlet temperature of the  DCMD system. The inlet temperatures of the system $T_{f}$ and $T_{p}$ in \eqref{glv1.2} are the measurements of the system. Moreover, the system is subject to unknown input heat flux disturbance $d(t)$ in the input of the system.
\subsection{State observer design}
We  design an extended state observer (ESO) that can estimate not only the state $w(t,x)$ of the controlled system \eqref{BBGL0.03}, but also the disturbance $d(t)$. The ESO is designed as follows;
\begin{equation}\label{BBGL0.04}
\begin{cases}
\displaystyle \partial_{t}\hat{w}(t,x)-\alpha\Delta \hat{w}(t,x)+B.\nabla \hat{w}(t,x)=0&\ t>0,\,\,\, x\in\Omega,\\
\displaystyle \frac{\partial \hat{w}(t,x)}{\partial \nu}=\mathcal{M}\hat{w}(t,x)&\ t>0,\,\,\, x\in\Gamma_{4},\\
\displaystyle \frac{\partial \hat{w}(t,x)}{\partial \nu}=u(t,x)& \ t>0,\,\,\, x\in \Gamma_{3},\\
\displaystyle \frac{\partial \hat{w}(t,x)}{\partial \nu}=0& \ t>0,\,\,\, x\in \Gamma_{2},\\
 \displaystyle \hat{w}(t,x)=y_{m}(t,x)& \ t>0,\,\,\, x\in \Gamma_{1},\\
 \displaystyle \hat{w}(0,x)=\hat{w}_{0}(x)&  x\in\Omega.
\end{cases}
\end{equation}
The error  $\tilde{w}(t,x)=\hat{w}(t,x)-w(t,x)$ is governed by
\begin{equation}\label{BBGL0.05}
\begin{cases}
\displaystyle \partial_{t}\tilde{w}(t,x)-\alpha\Delta \tilde{w}(t,x)+B.\nabla \tilde{w}(t,x)=0 &\ t>0,\,\,\, x\in\Omega,\\
\displaystyle \frac{\partial \tilde{w}(t,x)}{\partial \nu}=\mathcal{M}\tilde{w}(t,x)&\ t>0,\,\,\, x\in\Gamma_{4},\\
 \displaystyle \frac{\partial \tilde{w}(t,x)}{\partial \nu}(t,x)=0& \ t>0,\,\,\, x\in \Gamma_{3},\\
\displaystyle \frac{\partial \tilde{w}(t,x)}{\partial \nu}=0& \ t>0,\,\,\, x\in \Gamma_{2},\\
\displaystyle \tilde{w}(t,x)=0& \ t>0,\,\,\, x\in \Gamma_{1},\\
 \displaystyle \tilde{w}(0,x)=\tilde{w}_{0}(x)&  x\in\Omega.
\end{cases}
\end{equation}
Now, we give some boundary estimation on the error system.
\begin{lemma}\label{lemma1}
Let us assume $\tilde{w}_{0}\in\mathbf{X}$, the system \eqref{BBGL0.05} admits a unique solution $\tilde{w}\in C(0,\infty;\mathbf{X})$ and  there exists  $\delta_{0}>0$ such that
\[
\|\tilde{w}(t)\|_{\mathbf{X}}\leqslant e^{-\delta_{0}t}\|\tilde{w}_{0}\|_{\mathbf{X}},~~\forall t>0.
\]
Moreover, for any $\tau>0$ there exists $M_{0}>0$ depending on $\tau$ and $\delta_{0}>0$ such that
\begin{equation}\label{L1}
  \big\|\frac{\partial \tilde{w}}{\partial \nu}\big\|_{\mathcal{H}_{\Gamma_{1}}}\leqslant M_{0} e^{-\delta_{0}t}\|\tilde{w}_{0}\|_{\mathbf{X}}, \quad
  \big\| \tilde{w}\big\|_{\mathcal{H}_{\Gamma_{3}}}\leqslant M_{0} e^{-\delta_{0}t}\|\tilde{w}_{0}\|_{\mathbf{X}},~~\forall t>\tau.
\end{equation}
\end{lemma}
\begin{proof}
Let's introduce the operator
\[
\tilde{\mathcal{A}}=\Delta \phi+B.\nabla \phi,\,\,\forall \phi\in D\left(\tilde{\mathcal{A}}\right),
 \]
 and
 \[
  D(\tilde{\mathcal{A}})=\big\{\phi\in \mathbf{Z};\,\displaystyle \frac{\partial \phi}{\partial \nu}\big|_{\Gamma_{4}}=\mathcal{M}\phi;\frac{\partial \phi}{\partial \nu}\big|_{\Gamma_{2}}=\frac{\partial \phi}{\partial \nu}\big|_{\Gamma_{3}}=\phi|_{\Gamma_{1}}=0\big\}
  \]
From section \ref{sec-exist-unic}, by \cite[Theorem 2.7 p.211]{pazysemigroups} it is easy to show that $\tilde{\mathcal{A}}$ generates an analytical semigroup $\tilde{S}(t)$ which implies that \eqref{BBGL0.05} admits a unique solution 
\[
\tilde{w}=\tilde{S}\tilde{w}_{0}\in C(0,\infty;\mathbf{X}).
\]
Now, we show that the semi-group $\tilde{S}$ is exponentially stable. Indeed, we introduce the following Lyapunov function
$$V_{0}(t)=\frac{1}{2}\alpha_{p}\gamma_{p}\int_{\Omega}\tilde{w}_{1}(t,x)dx+\frac{1}{2}\alpha_{f}\gamma_{f}\int_{\Omega}\tilde{w}_{2}(t,x)dx.$$
Differentiating $V_{0}$ along the solution of \eqref{BBGL0.05} and using Green's formula and Poincar\'{e}e inequality we deduce that there exists $\delta>0$ such that
\[
\dot{V_{0}}(t)\leqslant -\delta V_{0}(t), 
\]
which gives the exponential stability of $\tilde{S}$, i.e.
\begin{equation}\label{L3}
  \|\tilde{w}(t)\|_{\mathbf{X}}= \|\tilde{S}(t)\tilde{w}_{0}\|_{\mathbf{X}}\leqslant e^{-\delta t} \|\tilde{w}_{0}\|_{\mathbf{X}}.
\end{equation}
Therefore, $\tilde{S}(t)$ is an analytic semigroup, \cite{pazysemigroups}, from any positive integer $m$, we have $\tilde{S}(t)\tilde{w}_{0}\in D(\tilde{\mathcal{A}})$ for all $t>0$ and there exists a constant $C>0$ such that
\begin{equation}\label{L4}
  \|\tilde{\mathcal{A}}\tilde{S}(t)\|\leqslant \frac{C}{t},\,\,\forall t>0.
\end{equation}
So, we obtain $\tilde{\mathcal{A}}^{m}\tilde{S}(t)\tilde{w}_{0}=\left(\tilde{\mathcal{A}}S\left(t/m\right)\right)^{m}\tilde{w}_{0} $, by \eqref{L4} it follows that
$$\begin{array}{ll}
    \|\Delta^{m}\tilde{w}(t,.)\|_{\mathbf{X}}& =\| \tilde{\mathcal{A}}^{m}\tilde{w}(t,.)\|_{\mathbf{X}}=\|(\tilde{\mathcal{A}}S(t/m))^{m}\tilde{w}_{0}\|_{\mathbf{X}}  \\
   &\displaystyle\leqslant \frac{C^{m}m^{m}}{t^{m}}\|\tilde{w}_{0}\|_{\mathbf{X}}.
\end{array}$$
Moreover, by using the Sobolev embedding theorem, there exists $C_{1}>0$ such that
$$\begin{array}{ll}
    \|\tilde{w}(t)\|_{\mathcal{H}_{s}^{2m}}& \leqslant C_{1}\left(\|\Delta^{m}\tilde{w}(t)\|_{\mathbf{X}}+\|\tilde{w}(t)\|_{\mathbf{X}}\right)  \\
   &\displaystyle\leqslant \left( \frac{C_{1}C^{m}m^{m}}{t^{m}}+C_{1}M_{0}e^{-\mu t}\right)\|\tilde{w}_{0}\|_{\mathbf{X}}.
\end{array}$$
The Sobolev trace theorem implies that
\begin{equation}\label{L1122}
  \big\|\frac{\partial \tilde{w}(t)}{\partial \nu}\big\|_{\mathcal{H}_{\Gamma_{3}}}\leqslant C_{3} \|\tilde{w}\|_{\mathcal{H}_{s}^{2m}},\quad  \big\|\frac{\partial \tilde{w}(t)}{\partial \nu}\big\|_{\mathcal{H}_{\Gamma_{1}}}\leqslant C_{3}\|\tilde{w}\|_{\mathcal{H}_{s}^{2m}},
\end{equation}
for some constant $C_{3}>0$. Therefore, by \eqref{L1} we get
\[
\begin{array}{ll}
    \|\Delta^{m}\tilde{w}(t)\|_{\mathbf{X}}& =\| \tilde{\mathcal{A}}^{m}\tilde{w}(t)\|_{\mathbf{X}}=\|\tilde{\mathcal{A}}^{m}\tilde{S}(t)\tilde{w}_{0}\|_{\mathbf{X}}  \\
   &\displaystyle=\|\tilde{S}(t-\tau)\tilde{\mathcal{A}}\tilde{S}(\tau)\tilde{w}_{0}\|_{\mathbf{X}} \\
    &\displaystyle\leqslant M e^{-\mu(t-\tau)}\|\tilde{\mathcal{A}}\tilde{S}(t)\tilde{w}_{0}\|_{\mathbf{X}},
\end{array}
\]
for any $\tau>0$. Finally, from \eqref{L1122} and the Sobolev embedding theorem, it follows the estimates in \eqref{L1}.
\end{proof}

\subsection{Servomechanism}
We design a servomechanism for system \eqref{BBGL0.03} in term of the reference signal $r(t,x)$.
\begin{equation}\label{BBGL0.06}
\begin{cases}
\displaystyle \partial_{t}v(t,x)-\alpha\Delta v(t,x)+B.\nabla v(t,x)=0  &\ t>0,\,\,\, x\in\Omega,\\
\displaystyle \frac{\partial v(t,x)}{\partial \nu}=\mathcal{M}v(t,x) &\ t>0,\,\,\, x\in\Gamma_{4},\\
\displaystyle v(t,x)=r(t)& \ t>0,\,\,\, x\in \Gamma_{3},\\
\displaystyle \frac{\partial v(t,x)}{\partial \nu}=0& \ t>0,\,\,\, x\in \Gamma_{2},\\
 \displaystyle v(t,x)=\hat{w}(t,x)& \ t>0,\,\,\, x\in \Gamma_{1},\\
 \displaystyle v(0,x)=v_{0}(x)&  x\in\Omega.
\end{cases}
\end{equation}
The error tracking is given by 
\begin{equation}\label{BBGL0.07}
\begin{aligned}
e(t,x)=y_{r}(t,x)-r(t,x)&={w(t,x)}_{|_{\Gamma_{3}}}-{v(t,x)}_{|_{\Gamma_{3}}}\\&=\left({w(t,x)}_{|_{\Gamma_{3}}}-{\hat{w}(t,x)}_{|_{\Gamma_{3}}}\right)+\left({\hat{w}(t,x)}_{|_{\Gamma_{3}}}-{v(t,x)}_{|_{\Gamma_{3}}}\right).
\end{aligned}
\end{equation}
Let us consider now the error equation $\eta(t,x)=v(t,x)-\hat{w}(t,x)$ between the state of servo-system and the state observer
\begin{equation}\label{BBGL0.071}
\begin{cases}
\displaystyle \partial_{t}\eta(t,x)-\alpha\Delta \eta (t,x)+B.\nabla \eta(t,x)=0  &\ t>0,\,\,\, x\in\Omega,\\
\displaystyle \frac{\partial \eta(t,x)}{\partial \nu}=\mathcal{M} \eta(t,x) &\ t>0,\,\,\, x\in\Gamma_{4},\\
\displaystyle \frac{\partial \eta (t,x)}{\partial \nu}=\frac{\partial v (t,x)}{\partial \nu}-u(t,x)& \ t>0,\,\,\, x\in \Gamma_{3},\\
\displaystyle \frac{\partial \eta (t,x)}{\partial \nu}=0& \ t>0,\,\,\, x\in \Gamma_{2},\\
 \displaystyle  \eta(t,x)=0& \ t>0,\,\,\, x\in \Gamma_{1},\\
 \displaystyle \eta (0,x)=\eta _{0}(x)&  x\in\Omega.
\end{cases}
\end{equation}
We assume that the output feedback control law
\begin{equation}\label{feedbacklaw1}
u_{1}(t,x)={\frac{\partial v_{1}(t,x)}{\partial \nu}}_{|_{\Gamma_{3}}}.
\end{equation}
Then, system \eqref{BBGL0.071} becomes
\begin{equation}\label{BBGL0.0721}
\begin{cases}
\displaystyle \partial_{t}\eta(t,x)-\alpha\Delta \eta (t,x)+B.\nabla \eta(t,x)=0 &\ t>0,\,\,\, x\in\Omega,\\
\displaystyle \frac{\partial \eta(t,x)}{\partial \nu}=\mathcal{M} \eta(t,x)&\ t>0,\,\,\, x\in\Gamma_{4},\\
 \displaystyle  \frac{\partial \eta(t,x)}{\partial \nu}=\left(0,\frac{\partial v_2 (t,x)}{\partial \nu}\right)& \ t>0,\,\,\, x\in \Gamma_{3},\\
\displaystyle \frac{\partial \eta (t,x)}{\partial \nu}=0& \ t>0,\,\,\, x\in \Gamma_{2},\\
\displaystyle \eta (t,x)=0& \ t>0,\,\,\, x\in \Gamma_{1},\\
 \displaystyle \eta (0,x)=\eta _{0}(x)&  x\in\Omega.
\end{cases}
\end{equation}
Now, we state the well-posedness and boundeness of the solution to system \eqref{BBGL0.06}.
\begin{lemma}\label{lemma31}
Suppose that $r\in W^{1,\infty}(0,\infty,\mathcal{H}_{\Gamma_{3}})$, $v_{0}\in\mathbf{X}$, then,  system \eqref{BBGL0.06} admits a unique solution $v\in C(0,\infty;\mathbf{X})$ which is uniformly bounded for all $t\geqslant 0$, i.e. $\sup_{t\geqslant 0}\|v(t)\|_{\mathbf{X}}<+\infty$. Moreover, for $r\in H^{1}(0,\infty, \mathcal{H}_{s\Gamma_{3}})$ we have $\lim_{t\to\infty}\|v(t)\|_{\mathbf{X}} =0$. 	
\end{lemma}
\begin{proof}
In order to give the well posedness we separate the servo system \eqref{BBGL0.06} into tow subsystem $\rho_{1}$ and $\rho_{2}$ described respectively by
\begin{equation}\label{eqv11}
\begin{cases}
\displaystyle \partial_{t}\rho_{1}(t,x)-\alpha\Delta \rho_{1}(t,x)+B.\nabla \rho_{1}(t,x)=0&\ t>0,\,\,\, x\in\Omega,\\
\displaystyle \frac{\partial \rho_{1}(t,x)}{\partial \nu}=\mathcal{M}\rho_{1}(t,x)&\ t>0,\,\,\, x\in\Gamma_{4},\\
\displaystyle \rho_{1}(t,x)=r(t,x)& \ t>0,\,\,\, x\in \Gamma_{3},\\
\displaystyle \frac{\partial \rho_{1}(t,x)}{\partial \nu}=0& \ t>0,\,\,\, x\in \Gamma_{2},\\
\displaystyle \rho_{1}(t,x)=0& \ t>0,\,\,\, x\in \Gamma_{1},\\
 \displaystyle \rho_{1}(0,x)=\rho_{10}(x)&  x\in\Omega,\\
\end{cases}
\end{equation}
and
\begin{equation}\label{eqv22}
\begin{cases}
\displaystyle \partial_{t}\rho_{2}(t,x)-\alpha\Delta \rho_{2}(t,x)+B.\nabla \rho_{2}(t,x)=0&\ t>0,\,\,\, x\in\Omega,\\
\displaystyle \frac{\partial \rho_{2}(t,x)}{\partial \nu}=\mathcal{M}\rho_{2}(t,x)&\ t>0,\,\,\, x\in\Gamma_{4},\\
\displaystyle \rho_{2}(t,x)=0& \ t>0,\,\,\, x\in \Gamma_{3},\\
\displaystyle \frac{\partial \rho_{2}(t,x)}{\partial \nu}=0& \ t>0,\,\,\, x\in \Gamma_{2},\\
\displaystyle \rho_{2}(t,x)=y_{m}(t)& \ t>0,\,\,\, x\in \Gamma_{1},\\
 \displaystyle \rho_{2}(0,x)=\rho_{20}(x)&  x\in\Omega.
\end{cases}
\end{equation}
First, we introduce the Dirichlet map $\Lambda_{1}\in \mathcal{L}(   \mathcal{H}_{\Gamma_{3}};   \mathcal{H}_{s}^{\frac{1}{2}})$ see \cite[pages 188 ]{lions2012non}
\begin{equation*}
\begin{cases}
\displaystyle -\alpha\Delta z+B.\nabla z=0&\  t>0,x\in\Omega,\\
\displaystyle \frac{\partial z}{\partial \nu}=\mathcal{M}z&\ t>0,\,\,\, x\in\Gamma_{4},\\
\displaystyle z=r(t,x)& \ t>0,\,\,\, x\in \Gamma_{3},\\
\displaystyle \frac{\partial z}{\partial \nu}=0& \ t>0, \,\,\, x\in \Gamma_{2},\\
\displaystyle z=0& \ t>0,\,\,\, x\in \Gamma_{1},\\
\end{cases}
\end{equation*}
for $r\in W^{1,\infty}(0,\infty;\mathcal{H}_{\Gamma_{3}})$ then $z\in W^{1,\infty}(0,\infty; \mathcal{H}_{s}^{\frac{1}{2}})$. Moreover, we have
\begin{equation*}
  \|z(t)\|_{\mathcal{H}_{s}^{\frac{1}{2}}}\leqslant C_{1}\|r(t)\|_{\mathcal{H}_{\Gamma_{3}}},\quad \|z_{t}(t)\|_{\mathcal{H}_{s}^{\frac{1}{2}}}\leqslant C_{2}\|r_t (t)\|_{\mathcal{H}_{\Gamma_{3}}},\,\,t\geqslant 0.
\end{equation*}
From the Sobolev embedding theorem, it follows that
\begin{equation}\label{sourcebound1}
  \|z_{t}(t)\|_{\mathbf{X}}\leqslant C_{2}\|r(t)\|_{\mathbf{X}}, \,\,t\geqslant 0.
\end{equation} 
Now, by the Dirichlet map $\Lambda_{1}$ we introduce $\bar{v}_{1}(t,x)=\rho_{1}(t,x)-z(t,x)$, then $\bar{v}_{1}$ satisfies
\begin{equation*}
\begin{cases}
\displaystyle \partial_{t}\bar{v}_{1}(t,x)-\alpha\Delta \bar{v}_{1}(t,x)+B.\nabla \bar{v}_{1}(t,x)=-z_{t}(t,x)&\ t>0,\,\,\, x\in\Omega,\\
\displaystyle \frac{\partial \bar{v}_{1}(t,x)}{\partial \nu}=\mathcal{M}\bar{v}_{1}(t,x)&\ t>0,\,\,\, x\in\Gamma_{4},\\
\displaystyle \bar{v}_{1}(t,x)=0& \ t>0,\,\,\, x\in \Gamma_{3},\\
\displaystyle \frac{\partial \bar{v}_{1}(t,x)}{\partial \nu}=0& \ t>0,\,\,\, x\in \Gamma_{2},\\
\displaystyle \bar{v}_{1}(t,x)=0& \ t>0,\,\,\, x\in \Gamma_{1},\\
 \displaystyle \bar{v}_{1}(0,x)=\bar{v}_{01}(x)&  x\in\Omega,\\
\end{cases}
\end{equation*}
can be written as
\begin{equation}\label{diffform1}
 \partial_{t}\bar{v}_{1}(t,.)=\bar{\mathcal{A}}\bar{v}_{1}(t,.)+B_{1}z_{t}(t,.),
\end{equation}

 where $B_{1}=-I$ and the operator $\bar{\mathcal{A}}$ is given by
 \[
\bar{\mathcal{A}}=\Delta \phi+B.\nabla \phi,\,\,\forall \phi\in D(\bar{\mathcal{A}}),
 \]
 and
 \[
  D(\bar{\mathcal{A}})=\big\{\phi\in \mathbf{Z};\,\displaystyle \frac{\partial \phi}{\partial \nu}\big|_{\Gamma_{4}}=\mathcal{M}\phi;\frac{\partial \phi}{\partial \nu}\big|_{\Gamma_{2}}=0;\phi|_{\Gamma_{1}}=\phi|_{\Gamma_{3}}=0\big\}.
  \]
 It is easy to show that $\bar{\mathcal{A}}$ generates an exponentially stable $C_{0}$-semigroup $e^{\bar{\mathcal{A}}}$. Then using \eqref{sourcebound1} and \eqref{diffform1} it follows that the system \eqref{eqv11} has a unique solution $\rho_{1}\in C(0,\infty; \mathbf{X})$. Moreover, using \cite[Lemma 1.1]{zhou2018performance}, we conclude that   
 \begin{equation}\label{Estv1}
 \displaystyle\lim_{t\rightarrow \infty}\|\rho_{1}(t)\|_{\mathbf{X}}=0.
 \end{equation}
 By the same argument we introduce $\Lambda_{2}: \mathcal{L}(   \mathcal{H}_{s\Gamma_{1}};   \mathcal{H}_{s}^{\frac{1}{2}})$. Then
  \begin{equation}\label{mapDir2}
\begin{cases}
\displaystyle -\alpha\Delta y+B.\nabla y=0&\  x\in\Omega,\\
\displaystyle \frac{\partial y}{\partial \nu}=\mathcal{M}.y&\ \,\,\, x\in\Gamma_{4},\\
\displaystyle y=0& \ \,\,\, x\in \Gamma_{3},\\
\displaystyle \frac{\partial y}{\partial \nu}=0& \ \,\,\, x\in \Gamma_{2},\\
\displaystyle y=y_{m}& \ \,\,\, x\in \Gamma_{1},\\
\end{cases}
\end{equation}
with   $y_{m}\in W^{1,\infty}(0,\infty;\mathcal{H}_{s\Gamma_{1}})$, then $y\in W^{1,\infty}(0,\infty;\mathcal{H}_{s}^{\frac{1}{2}})$. Finally, the system \eqref{eqv22} has a unique solution, $\rho_{2}\in C(0,\infty;\mathbf{X})$ and by  \cite[Lemma 1.1]{zhou2018performance}, we have 
 \begin{equation}\label{Estv2}
  \lim_{t\rightarrow\infty}\|\rho_{2}(t)\|_{\mathbf{X}}=0.
 \end{equation}
  From \eqref{Estv1} and \eqref{Estv2}, we have 
 \begin{equation}
  \lim_{t\rightarrow\infty}\|v(t)\|_{\mathbf{X}}=0.
 \end{equation}
  This finish the proof.
\end{proof}

\begin{lemma}\label{lemma21}
Let us assume $\eta_{0}\in\mathbf{X}$, the system \eqref{BBGL0.05} admits a unique solution $\eta\in C(0,\infty;\mathbf{X})$ such that
$$\|\eta(t)\|_{\mathbf{X}}\leqslant e^{-t}\|\eta_{0}\|_{\mathbf{X}}+C\|v(t)\|_{\mathbf{X}},~~\forall t>0,$$
for a constant $C>0$. Moreover, for any $\tau>0$ there exist $M_{0}>0$ depending on $\tau$ and $\delta_{0}>0$  such that
\begin{equation}\label{L1121}
  \big\|\frac{\partial \eta}{\partial \nu}\big\|_{\mathcal{H}_{\Gamma_{1}}}\leqslant M_{0} e^{-\delta_{0}t}\|\eta_{0}\|_{\mathbf{X}}+C_{1}\|v(t)\|_{\mathbf{X}},\quad \big\| \eta\big\|_{\mathcal{H}_{\Gamma_{3}}}\leqslant M_{0} e^{-\delta_{0}t}\|\eta_{0}\|_{\mathbf{X}}+C_{1}\|v(t)\|_{\mathbf{X}},
\end{equation}
for all $t>\tau$ and for a positive constant $C_{1}$.
\end{lemma}
\begin{proof}
The proof is similar to the proof of Lemmas \ref{lemma1} and \ref{lemma31}
\end{proof}
Now, we are able to give an asymptotical estimation of the disturbance $d(t)$. From the error system, we obtain the following
\[
\frac{\partial \tilde{w}}{\partial \nu}(t,x)=\frac{\partial \hat{w}}{\partial \nu}(t,x)-\frac{\partial w}{\partial \nu}(t,x)=\frac{\partial \hat{w}}{\partial \nu}(t,x)-d(t)\quad \forall t\geqslant 0,\,\,\,x\in \Gamma_{1}.
\]
Consequently, from \eqref{L1} and \eqref{L1121}, the disturbance can be estimated by 
\begin{equation}\label{distApp}
d\approx\frac{\partial \hat{w}}{\partial \nu},\quad\text{for a large value of}\,\, t.
\end{equation}

\subsection{Control design and closed-loop system}
Now, we analyze the performance output tracking to  the closed loop system of \eqref{BBGL0.03}. Using the feedback law \eqref{feedbacklaw1},  the closed loop is given by 
\begin{equation}\label{closedloopsystem1}
\begin{cases}
\displaystyle \partial_{t}w(t,x)-\alpha\Delta w(t,x)+B.\nabla w(t,x)=0&\ t>0,\,\,\, x\in\Omega,\\
\displaystyle \frac{\partial w(t,x)}{\partial \nu}=\mathcal{M}w(t,x)&\ t>0,\,\,\, x\in\Gamma_{4},\\
\displaystyle \frac{\partial w(t,x)}{\partial \nu}=\left(\frac{\partial v_1 (t,x)}{\partial \nu},0\right)&\ t>0,\,\,\, x\in \Gamma_{3},\\
\displaystyle \frac{\partial w(t,x)}{\partial \nu}=0& \ t>0,\,\,\, x\in \Gamma_{2},\\
\displaystyle \frac{\partial w(t,x)}{\partial \nu}=d(t)&\ t>0,\,\,\, x\in \Gamma_{1},\\
 \displaystyle \partial_{t}\hat{w}(t,x)-\alpha\Delta \hat{w}(t,x)+B.\nabla \hat{w}(t,x)=0&\ t>0,\,\,\, x\in\Omega,\\
\displaystyle \frac{\partial \hat{w}(t,x)}{\partial \nu}=\mathcal{M}\hat{w}(t,x)&\ t>0,\,\,\, x\in\Gamma_{4},\\
\displaystyle \frac{\partial \hat{w}(t,x)}{\partial \nu}=\left(\frac{\partial v_1 (t,x)}{\partial \nu},0\right)& \ t>0,\,\,\, x\in \Gamma_{3},\\
\displaystyle \frac{\partial \hat{w}(t,x)}{\partial \nu}=0& \ t>0,\,\,\, x\in \Gamma_{2},\\
 \displaystyle \hat{w}(t,x)=y_{m}(t,x)& \ t>0,\,\,\, x\in \Gamma_{1},\\
\displaystyle \partial_{t}v(t,x)-\alpha\Delta v(t,x)+B.\nabla v(t,x)=0  &\ t>0,\,\,\, x\in\Omega,\\
\displaystyle \frac{\partial v(t,x)}{\partial \nu}=\mathcal{M}v(t,x) &\ t>0,\,\,\, x\in\Gamma_{4},\\
\displaystyle v(t,x)=r(t)& \ t>0,\,\,\, x\in \Gamma_{3},\\
\displaystyle \frac{\partial v(t,x)}{\partial \nu}=0& \ t>0,\,\,\, x\in \Gamma_{2},\\
\displaystyle v(t,x)=\hat{w}(t,x)& \ t>0,\,\,\, x\in \Gamma_{1},\\
\displaystyle w(0,x)=w_{0}(x)&  x\in\Omega,\\
\displaystyle \hat{w}(0,x)=\hat{w}_{0}(x)&  x\in\Omega,\\
\displaystyle v(0,x)=v_{0}(x)&  x\in\Omega.
\end{cases}
\end{equation}

\begin{thm}
Let $r\in W^{1,\infty}(0,\infty,\mathcal{H}_{\Gamma_{3}})$,  for any initial value $(w_{0},\hat{w}_{0},v_{0})\in \mathbf{X}^{3}$, the closed-loop system \eqref{closedloopsystem1} admits a unique solution $(w,\hat{w},v) \in C(0; \infty;\mathbf{X}^{3})$ such that 
\begin{enumerate}
\item \[\sup_{t\geqslant 0}\left(\|\hat{w}(t)\|_{\mathbf{X}}+\|w(t)\|_{\mathbf{X}}+\|v(t)\|_{\mathbf{X}}\right)<+\infty, \quad \forall t \geqslant 0\]

\item There exists a constant $M$ depending on the initial data $\left(w_{0}, \hat{w}_{0}\right)$ and $\mu>0$ such that 
\[\|w(t)-\hat{w}(t)\|_{\mathbf{X}}\le Me^{-\mu t},\quad \forall t\geqslant 0,\]
and for any $\tau>0$, we have
\[\|e(t)\|_{\mathbf{X}}\le M_{1}e^{-\mu_{1} t},\quad \forall t\geqslant 0,\quad \forall t\geqslant 0,\] 
where $M_{1}$ depending on the initial data $\left(w_{0}, \hat{w}_{0},v_{0}\right)$ and a constant $\mu_{1}>0$. 

\item Moreover, for $r\in H^{1}(0,\infty,\mathcal{H}_{\Gamma_{3}})$ we have the asymptotically stability of system \eqref{closedloopsystem1}
\begin{equation}\label{EstimationCnvergence1}
\lim_{t\to\infty}\left(\|\hat{w}(t)\|_{\mathbf{X}}+\|w(t)\|_{\mathbf{X}}+\|v(t)\|_{\mathbf{X}}\right)=0.
\end{equation}
\end{enumerate}
\end{thm}
\begin{proof}
Let $\tilde{w}(t,x)=\hat{w}(t,x)-w(t,x)$ and $\eta(t,x)=v(t,x)-\hat{w}(t,x)$, then the closed loop system \eqref{closedloopsystem1} is equivalent to

\begin{equation}\label{BBGaa1}
\begin{cases}
\displaystyle \partial_{t}\tilde{w}(t,x)-\alpha\Delta \tilde{w}(t,x)+B.\nabla \tilde{w}(t,x)=0 &\ t>0,\,\,\, x\in\Omega,\\
\displaystyle \frac{\partial \tilde{w}(t,x)}{\partial \nu}=\mathcal{M}\tilde{w}(t,x)&\ t>0,\,\,\, x\in\Gamma_{4},\\
 \displaystyle \frac{\partial \tilde{w}(t,x)}{\partial \nu}(t,x)=0& \ t>0,\,\,\, x\in \Gamma_{3},\\
\displaystyle \frac{\partial \tilde{w}(t,x)}{\partial \nu}=0& \ t>0,\,\,\, x\in \Gamma_{2},\\
\displaystyle \tilde{w}(t,x)=0& \ t>0,\,\,\, x\in \Gamma_{1},\\
 \displaystyle \partial_{t}\eta(t,x)-\alpha\Delta \eta(t,x)+B.\nabla \eta(t,x)=0 &\ t>0,\,\,\, x\in\Omega,\\
\displaystyle \frac{\partial \eta(t,x)}{\partial \nu}=\mathcal{M}\eta(t,x)&\ t>0,\,\,\, x\in\Gamma_{4},\\
 \displaystyle \frac{\partial \eta(t,x)}{\partial \nu}(t,x)=\left(0,\frac{\partial v_{2} (t,x)}{\partial \nu}\right)& \ t>0,\,\,\, x\in \Gamma_{3},\\
\displaystyle \frac{\partial \eta(t,x)}{\partial \nu}=0& \ t>0,\,\,\, x\in \Gamma_{2},\\
\displaystyle \eta(t,x)=0& \ t>0,\,\,\, x\in \Gamma_{1},\\
\displaystyle \partial_{t}\rho_{2}(t,x)-\alpha\Delta \rho_{2}(t,x)+B.\nabla \rho_{2}(t,x)=0&\ t>0,\,\,\, x\in\Omega,\\
\displaystyle \frac{\partial \rho_{2}(t,x)}{\partial \nu}=\mathcal{M}\rho_{2}(t,x)&\ t>0,\,\,\, x\in\Gamma_{4},\\
\displaystyle \rho_{2}(t,x)=0& \ t>0,\,\,\, x\in \Gamma_{3},\\
\displaystyle \frac{\partial \rho_{2}(t,x)}{\partial \nu}=0& \ t>0,\,\,\, x\in \Gamma_{2},\\
\displaystyle \rho_{2}(t,x)=0& \ t>0,\,\,\, x\in \Gamma_{1},\\
\displaystyle \tilde{w}(0,x)=\tilde{w}_{0}(x)&  x\in\Omega,\\
\displaystyle \eta(0,x)=\eta_{0}(x)&  x\in\Omega,\\
\displaystyle \rho_{2}(0,x)=\rho_{20}(x)&  x\in\Omega,
\end{cases}
\end{equation}
and  
\begin{equation*}
\begin{cases}
\displaystyle \partial_{t}\rho_{1}(t,x)-\alpha\Delta \rho_{1}(t,x)+B.\nabla \rho_{1}(t,x)=0&\ t>0,\,\,\, x\in\Omega,\\
\displaystyle \frac{\partial \rho_{1}(t,x)}{\partial \nu}=\mathcal{M}\rho_{1}(t,x)&\ t>0,\,\,\, x\in\Gamma_{4},\\
\displaystyle \rho_{1}(t,x)=r(t,x)& \ t>0,\,\,\, x\in \Gamma_{3},\\
\displaystyle \frac{\partial \rho_{1}(t,x)}{\partial \nu}=0& \ t>0,\,\,\, x\in \Gamma_{2},\\
\displaystyle \rho_{1}(t,x)=y_{m}(t,x)& \ t>0,\,\,\, x\in \Gamma_{1},\\
 \displaystyle \rho_{1}(0,x)=\rho_{10}(x)&  x\in\Omega.
\end{cases}
\end{equation*}
Thus system \eqref{BBGaa1} becomes
\begin{equation}\label{BBGad}
\begin{cases}
 \displaystyle \partial_{t}\eta(t,x)-\alpha\Delta \eta(t,x)+B.\nabla \eta(t,x)=0 &\ t>0,\,\,\, x\in\Omega,\\
\displaystyle \frac{\partial \eta(t,x)}{\partial \nu}=\mathcal{M}\eta(t,x)&\ t>0,\,\,\, x\in\Gamma_{4},\\
 \displaystyle \frac{\partial \eta(t,x)}{\partial \nu}(t,x)=\left(0,\frac{\partial \rho_{12} (t,x)}{\partial \nu}+\frac{\partial \rho_{22} (t,x)}{\partial \nu}\right)& \ t>0,\,\,\, x\in \Gamma_{3},\\
\displaystyle \frac{\partial \eta(t,x)}{\partial \nu}=0& \ t>0,\,\,\, x\in \Gamma_{2},\\
\displaystyle \eta(t,x)=0& \ t>0,\,\,\, x\in \Gamma_{1},\\
\displaystyle \partial_{t}\rho_{2}(t,x)-\alpha\Delta \rho_{2}(t,x)+B.\nabla \rho_{2}(t,x)=0&\ t>0,\,\,\, x\in\Omega,\\
\displaystyle \frac{\partial \rho_{2}(t,x)}{\partial \nu}=\mathcal{M}\rho_{2}(t,x)&\ t>0,\,\,\, x\in\Gamma_{4},\\
\displaystyle \rho_{2}(t,x)=0& \ t>0,\,\,\, x\in \Gamma_{3},\\
\displaystyle \frac{\partial \rho_{2}(t,x)}{\partial \nu}=0& \ t>0,\,\,\, x\in \Gamma_{2},\\
\displaystyle \rho_{2}(t,x)=0& \ t>0,\,\,\, x\in \Gamma_{1},\\
\displaystyle \eta(0,x)=\eta_{0}(x)&  x\in\Omega,\\
\displaystyle \rho_{2}(0,x)=\rho_{20}(x)&  x\in\Omega.
\end{cases}
\end{equation}
System \eqref{BBGad} can be  rewritten as follows 
\begin{equation}
\frac{d}{dt}\left(\eta(t,.), \rho_{2}(t,.)\right)=\mathbb{A}\left(\eta(t,.), \rho_{2}(t,.)\right)+ \mathbb{B}\frac{\partial \rho_{1}(t,.)}{\partial \nu}.
\end{equation}
Now, we need to show that $\mathbb{B}v_{1}=\left(0,{\frac{\partial v_{12}(t,.)}{\partial \nu}}_{|_{\Gamma_{1}}}\right)$ , for all $v=v_{1}+v_{2}\in \bar{\mathcal{A}}$, is an admissible operator for $e^{\bar{\mathcal{A} t}}$.

Let 
\[
\mathbb{A}=\left(\alpha\Delta \phi -B.\nabla \phi,\alpha\Delta \psi -B.\nabla \psi\right),\,\,\forall \left(\phi,\psi\right)\in D(\mathbb{A}),
 \]
 and
 \begin{equation}
 \begin{aligned}
  D(\mathbb{A})&=\Big\{\left(\phi,\psi\right)\in \mathbf{Z};\,\displaystyle \frac{\partial \phi}{\partial \nu}\big|_{\Gamma_{4}}=\mathcal{M}\phi;\frac{\partial \phi}{\partial \nu}\big|_{\Gamma_{2}}=0;\phi|_{\Gamma_{1}}=0;\\&{\frac{\partial \phi(t,x)}{\partial \nu}}_{\Gamma_{3}}={\left(0,\frac{\partial \psi_{2} (t,x)}{\partial \nu}\right)}_{\Gamma_{3}}
  \,\displaystyle \frac{\partial \psi}{\partial \nu}\big|_{\Gamma_{4}}=\mathcal{M}\psi;\frac{\partial \psi}{\partial \nu}\big|_{\Gamma_{2}}=0;\psi|_{\Gamma_{1}}=\psi|_{\Gamma_{3}}=0 \Big\}.
  \end{aligned}
 \end{equation}
Let 
 \[
\mathbb{A}^{\ast}=\left(\alpha\Delta \eta +\nabla.\left(B\phi\right),\alpha\Delta \psi+\nabla.\left(B\psi\right)\right),\,\,\forall \left(\phi,\psi\right)\in D(\mathbb{A}^{\ast}),
 \]
 \begin{equation}
 \begin{aligned}
  D(\mathbb{A}^{\ast})&=\Big\{\left(\phi,\psi\right)\in \mathbf{Z};\,\displaystyle \frac{\partial \phi}{\partial \nu}\big|_{\Gamma_{4}}=\mathcal{M}^{\ast}\phi;\frac{\partial \phi}{\partial \nu}\big|_{\Gamma_{2}}=0;\phi|_{\Gamma_{1}}=0; \phi_{|_{\Gamma_{3}}}=\left(0,{\psi_{2}}_{|_{\Gamma_{3}}}\right)\\&
  \,\displaystyle \frac{\partial \psi}{\partial \nu}\big|_{\Gamma_{4}}=\mathcal{M}^{\ast}\psi;\frac{\partial \psi}{\partial \nu}\big|_{\Gamma_{2}}=0;\psi|_{\Gamma_{1}}=\psi|_{\Gamma_{3}}=0 \Big\}.
  \end{aligned}
 \end{equation}
 
 $\mathcal{M}^{\ast}=\begin{bmatrix}
   \gamma_{f}&-\gamma_{p}\\ \gamma_{f}&-\gamma_{p}
\end{bmatrix}$, $\mathbb{B}$ is an observation operator for the adjoint semigroup $e^{\mathbb{A}^{\ast}t}$. So we use to fact  that $\mathbb{B}^{\ast}\mathbb{A}^{\ast}$ is bounded from $\mathbf{Z}$ to $\mathcal{H}^{2}_{\Gamma_{1}}$. The adjoint problem is given by
 \begin{equation}\label{BBGad}
\begin{cases}
 \displaystyle \partial_{t}\eta^{\ast}(t,x)-\alpha\Delta \eta^{\ast}(t,x)-\nabla.\left(B\eta^{\ast}(t,x)\right)=0 &\ t>0,\,\,\, x\in\Omega,\\
\displaystyle \frac{\partial \eta^{\ast}(t,x)}{\partial \nu}=\mathcal{M}^{\ast}\eta^{\ast}(t,x)&\ t>0,\,\,\, x\in\Gamma_{4},\\
 \displaystyle  \eta^{\ast}(t,x)=\left(0, \rho^{\ast}_{22} (t,x)\right)& \ t>0,\,\,\, x\in \Gamma_{3},\\
\displaystyle \frac{\partial \eta^{\ast}(t,x)}{\partial \nu}=0& \ t>0,\,\,\, x\in \Gamma_{2},\\
\displaystyle \eta^{\ast}(t,x)=0& \ t>0,\,\,\, x\in \Gamma_{1},\\
\displaystyle \partial_{t}\rho^{\ast}_{2}(t,x)-\alpha\Delta \rho^{\ast}_{2}(t,x)-\nabla.\left(B\rho^{\ast}(t,x)\right)=0&\ t>0,\,\,\, x\in\Omega,\\
\displaystyle \frac{\partial \rho^{\ast}_{2}(t,x)}{\partial \nu}=\mathcal{M}^{\ast}\rho^{\ast}_{2}(t,x)&\ t>0,\,\,\, x\in\Gamma_{4},\\
\displaystyle \rho^{\ast}_{2}(t,x)=0& \ t>0,\,\,\, x\in \Gamma_{3},\\
\displaystyle \frac{\partial \rho^{\ast}_{2}(t,x)}{\partial \nu}=0& \ t>0,\,\,\, x\in \Gamma_{2},\\
\displaystyle \rho^{\ast}_{2}(t,x)=0& \ t>0,\,\,\, x\in \Gamma_{1},\\
\displaystyle \eta^{\ast}(0,x)=\eta^{\ast}_{0}(x)&  x\in\Omega,\\
\displaystyle \rho^{\ast}_{2}(0,x)=\rho^{\ast}_{20}(x)&  x\in\Omega.
\end{cases}
\end{equation}

We need to show that for every $\tau^{\ast}$ there exists $C=C_{\tau^{\ast}}$ such that 
\[
\int_{0}^{\tau^{\ast}}\int_{\Gamma_{1}}{\rho_{2}^{\ast}}^{2}(t,x)dt dx\leqslant C\left(\|\rho_{20}^{\ast}\|^{2}_{\mathbf{X}}+\|\eta_{0}^{\ast}\|^{2}_{\mathbf{X}}\right).
\]
By the Sobolev embedding theorem and the trace theorem,  we deduce that $\mathbb{B}$ is admissible for $e^{\mathbb{A}t}$. From Lemmas \ref{lemma1} and \ref{lemma21} for any $\tau>0$, there exist $M_{1}$ and $M_{2}$ depending on $\tilde{w}_{01}$ and $\mu_{0}$ and $\mu_{1},\,\mu_{2}>0$
\begin{equation}\label{eq11}
  \|\tilde{w}(t)\|_{\mathbf{X}}+\|\eta(t)\|_{\mathbf{X}}\leqslant M_{1}e^{-\mu_{1}t},\,t\geqslant0,
\end{equation}
\begin{equation}\label{eq21}
  \|\tilde{w}(t)\|_{\mathcal{H}_{\Gamma_{3}}}+\|\eta(t)\|_{\mathcal{H}_{\Gamma_{3}}}\leqslant M_{2}e^{-\mu_{2}t},\,t\geqslant \tau.
\end{equation}
The $(w,\hat{w},v)$-system is well-posed and admits a unique solution. From Lemma  \ref{lemma1}, \eqref{eq11} and \eqref{eq21} we conclude claim 1. Moreover, claim 2 follow from \eqref{eq11} and \eqref{eq21}. So, we have $e(t,x)=y_{0}(t,x)-r(t,x)=-\tilde{w}(t,x)|_{\Gamma_{3}}-\eta(t,x)|_{\Gamma_{3}}$, then the claim 3 follows from \eqref{eq21}. Finally, by Lemma \ref{lemma31} and \eqref{eq11}, estimate \eqref{EstimationCnvergence1} holds as well.
\end{proof}
Now, we discuss the performance of our tracking problem for some noise measurement $\sigma$.
\begin{proposition}
Suppose that $r\in H^{1}(0,\infty,\mathcal{H}_{\Gamma_{3}})$, assuming $y_{m}={w}_{|_{\Gamma_{1}}}+\sigma$ where $\sigma$ is the noise $\sigma\in W^{2,\infty}(0,\infty,\mathcal{H}_{\Gamma_{1}})$. Then, for $\left(w_{0}, \hat{w}_{0},v_{0}\right)$, the closed-loop system \eqref{closedloopsystem1} admits a unique solution $\left(w,\hat{w},v\right)\in C(0,\infty;\mathbf{X}^{3})$. The output tracking is robust with respect to $\sigma$. In addition,  for any fixed $\tau>0$, there exists two constants $M> 0$ which depend only on initial value and a constant $\mu>0$  such that
\[
\|e(t)\|_{\mathbf{X}}\le M\left(e^{-\mu t}+ \|\sigma\|_{W^{2,\infty}(0,\infty,\mathcal{H}_{\Gamma_{1}})}\right),\quad \quad \forall t\geqslant 0.
\]
\end{proposition}
\begin{proof}
Let $\tilde{w}(t,x)=\hat{w}(t,x)-w(t,x)$ then the error of of estimation is governed by
\begin{equation}\label{tildewnoise}
\begin{cases}
\displaystyle \partial_{t}\tilde{w}(t,x)-\alpha\Delta \tilde{w}(t,x)+B.\nabla \tilde{w}(t,x)=0 &\ t>0,\,\,\, x\in\Omega,\\
\displaystyle \frac{\partial \tilde{w}(t,x)}{\partial \nu}=\mathcal{M}\tilde{w}(t,x)&\ t>0,\,\,\, x\in\Gamma_{4},\\
 \displaystyle \frac{\partial \tilde{w}(t,x)}{\partial \nu}(t,x)=0& \ t>0,\,\,\, x\in \Gamma_{3},\\
\displaystyle \frac{\partial \tilde{w}(t,x)}{\partial \nu}=0& \ t>0,\,\,\, x\in \Gamma_{2},\\
\displaystyle \tilde{w}(t,x)=\sigma (t,x)& \ t>0,\,\,\, x\in \Gamma_{1},\\
 \displaystyle \tilde{w}(0,x)=\tilde{w}_{0}(x)&  x\in\Omega.
\end{cases}
\end{equation}
Now, in order to show the well-posedness of $(\tilde{w},\eta,v)$, we will process as previously by introducing the Dirichlet map  $\Lambda_{3}\in \mathcal{L}(   \mathcal{H}_{\Gamma_{3}};   \mathcal{H}_{s}^{\frac{1}{2}})$, $\Lambda_{3}(\sigma)=z$ where is the solution of
\begin{equation*}
\begin{cases}
\displaystyle -\alpha\Delta z+B.\nabla z=0&\  x\in\Omega,\\
\displaystyle \frac{\partial z}{\partial \nu}=\mathcal{M}.z&\ \,\,\, x\in\Gamma_{4},\\
\displaystyle \frac{\partial z}{\partial \nu}=0& \ \,\,\, x\in \Gamma_{3},\\
\displaystyle \frac{\partial z}{\partial \nu}=0& \ \,\,\, x\in \Gamma_{2},\\
\displaystyle z=\sigma& \ \,\,\, x\in \Gamma_{1},\\
\end{cases}
\end{equation*}
we have $\sigma\in W^{2,\infty}(0,\infty;\mathcal{H}_{\Gamma_{1}})$ then it follows that $z\in W^{2,\infty}(0,\infty;\mathcal{H}_{s}^{1/2})$ and for all $t\geqslant 0$, we have 
\begin{equation*}
  \|z(t)\|_{\mathcal{H}_{s}^{1/2}}\leqslant C_{1}\|\sigma(t)\|_{\mathcal{H}_{\Gamma_{1}}},\quad \|z_{t}(t)\|_{\mathcal{H}_{s}^{1/2}}\leqslant C_{2}\|\sigma_{t}(t)\|_{\mathcal{H}_{\Gamma_{1}}}.
\end{equation*}
Then, we can deduce there exist $C>0$
\begin{equation*}
  \|z_{t}(t)\|_{L^{2}(\Omega)}\leqslant C\|\sigma\|_{W^{2,\infty}(0,\infty;\mathcal{H}_{\Gamma_{1}})}, \quad \forall t \geqslant 0.
\end{equation*}
Thus, using the Dirichlet map $\Lambda_{3}$ and introducing the new variable  $\bar{w}(t,x)=\tilde{w}(t,x)-z(t,x)$,  system \eqref{tildewnoise} becomes 
\begin{equation*}
\begin{cases}
\displaystyle \partial_{t}\bar{w}(t,x)-\alpha\Delta \bar{w}(t,x)+B.\nabla \bar{w}(t,x)=-\partial_{t}z &\ t>0,\,\,\, x\in\Omega,\\
\displaystyle \frac{\partial \bar{w}(t,x)}{\partial \nu}=\mathcal{M}\bar{w}(t,x)&\ t>0,\,\,\, x\in\Gamma_{4},\\
 \displaystyle \frac{\partial \bar{w}(t,x)}{\partial \nu}(t,x)=0& \ t>0,\,\,\, x\in \Gamma_{3},\\
\displaystyle \frac{\partial \bar{w}(t,x)}{\partial \nu}=0& \ t>0,\,\,\, x\in \Gamma_{2},\\
\displaystyle \bar{w}(t,x)=0& \ t>0,\,\,\, x\in \Gamma_{1},\\
 \displaystyle \bar{w}(0,x)=\bar{w}_{0}(x)&  x\in\Omega.
\end{cases}
\end{equation*}
We know that $\tilde{\mathcal{A}}$ defined in the proof of Lemma \ref{lemma1} generates our analytic semigroup $\tilde{S}(t)$, thus $\bar{w}_{0}\in \mathbf{X}$, the mild solution of \eqref{tildewnoise} is given by 
\[
\bar{w}(t,x)=S(t)\bar{w}_{0}-\int_{0}^{t}S(t-s)z_{s}(s,.)ds.
\]
Using the parabolic regularity and then the nonhomogenous parabolic equation see \cite{TucWeiss}, we can deduce 
\[
\|\bar{w}(t)\|_{\mathbf{Z}}\leqslant C \left(e^{-\mu_{0}t}\|\bar{w}_{0}\|_{L^{2}(\Omega)}+\|\sigma\|_{W^{2,\infty}(0,\infty;\mathcal{H}_{\Gamma_{1}})}\right), \quad \forall t \geqslant 0,
\]
where $C>0$. Then, by  the trace theorem, it follows that 
\[
\|\bar{w}(t)\|_{\mathcal{H}_{\Gamma_{1}}}\leqslant C_{1}\left(e^{-\mu_{0}t}\|\bar{w}_{0}\|_{L^{2}(\Omega)}+\|\sigma\|_{W^{2,\infty}(0,\infty;\mathcal{H}_{\Gamma_{1}})}\right), \quad \forall t \geqslant 0.
\]
Thus finished the proof. 
\end{proof}

\section{Numerical illustrations}\label{sectionnumeric}
In this section, we will show a  numerical example illustrating the effectiveness of the feedback  control law. In order to simplify the notations we assume that $L=2$, $\alpha_{f}=3$, $\alpha_{p}=3.5$, $\gamma_{f}=0.2$, $\gamma_{p}=0.1$ and the convection term $B$ is assumed to be zero. For numerical computations, we take  $\Omega=[0,1]\times[0,2]$. The external disturbance and the reference signal are taken as $d(t)=\left(0.1\sin(\frac{\pi}{2} t),0.1\sin(\frac{\pi}{2} t)\right)$, and $r(t,x)=\left(15\sin(\frac{\pi}{2} x t),10\sin(\frac{\pi}{2} xt)\right)$. The initial values are taken as 

\begin{equation*}
\begin{cases}
\displaystyle w_{0}(x,y)=\left(6\sin(\pi x)\cos(\frac{\pi}{4}y),3\sin(\frac{\pi}{2} x)\cos(\frac{\pi}{4}y)\right)&\ \, (x,y)\in\Omega,\\
\displaystyle \hat{w}_{0}(x,y)=\left(5\sin(\pi x)\cos(\frac{\pi}{8}y),2.5\sin(\pi x)\cos(\frac{\pi}{4}y)\right)&\ \,  (x,y)\in\Omega,\\
\displaystyle v_{0}(x,y)=\left(4\sin(\frac{\pi}{2} x)\cos(\frac{\pi}{4}y),3.5\sin(\frac{\pi}{2} x)\cos(\frac{\pi}{8}y)\right)&\ \,  (x,y)\in\Omega.
\end{cases}
\end{equation*}
The backward Euler method in time and the finite element method for the spatial domain are used to discretize the system \eqref{closedloopsystem1}. The grid sizes are  $h=1/100$ for space step and $dt=2\times 10^{-3}$ for the time step. The numerical algorithm is programmed in Freefem++\footnote{https://freefem.org}, and the numerical results are showed in figures  \ref{figsum1}-\ref{figsum2}. Figure \ref{figsum1} plots the $L^{2}$ norm of the tracking error  $\|e(t)\|_{\mathcal{H}_{\Gamma_{3}}}=\|y_{r}(t)-r(t)\|_{\mathcal{H}_{\Gamma_{3}}}$ for the output signal to be regulated and the reference. Figure \ref{figsum2} shows the $L^{2}$ norm of $w-v$ and $\hat{w}-v$ systems at final time $t=10$. 

\begin{figure}[!hbt]
	\center	\includegraphics[scale=0.20]{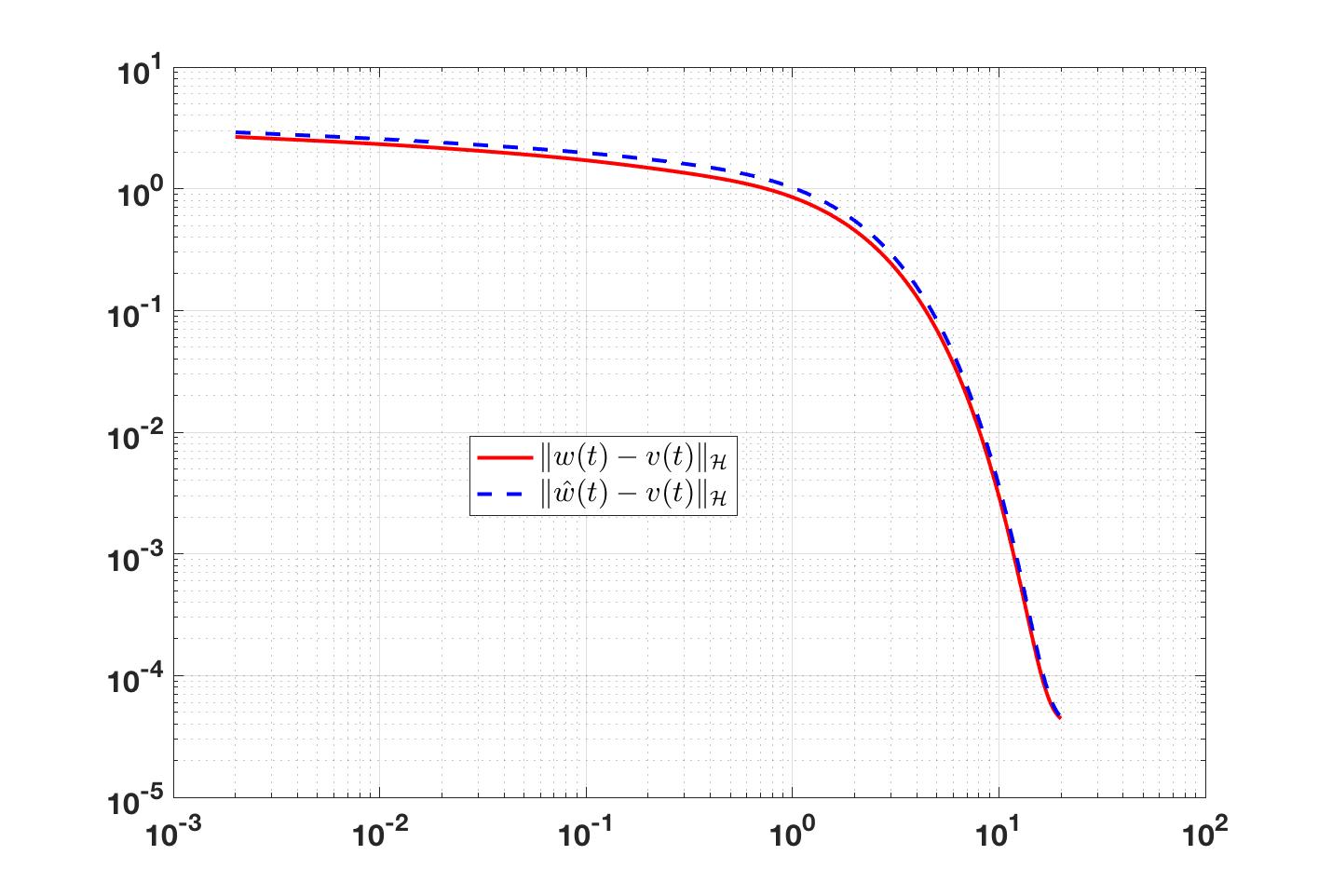}
	\caption{The error of estimation for the closed-loop system}
	\label{figsum1}
\end{figure}

\begin{figure}[!hbt]
		\center\includegraphics[scale=0.17]{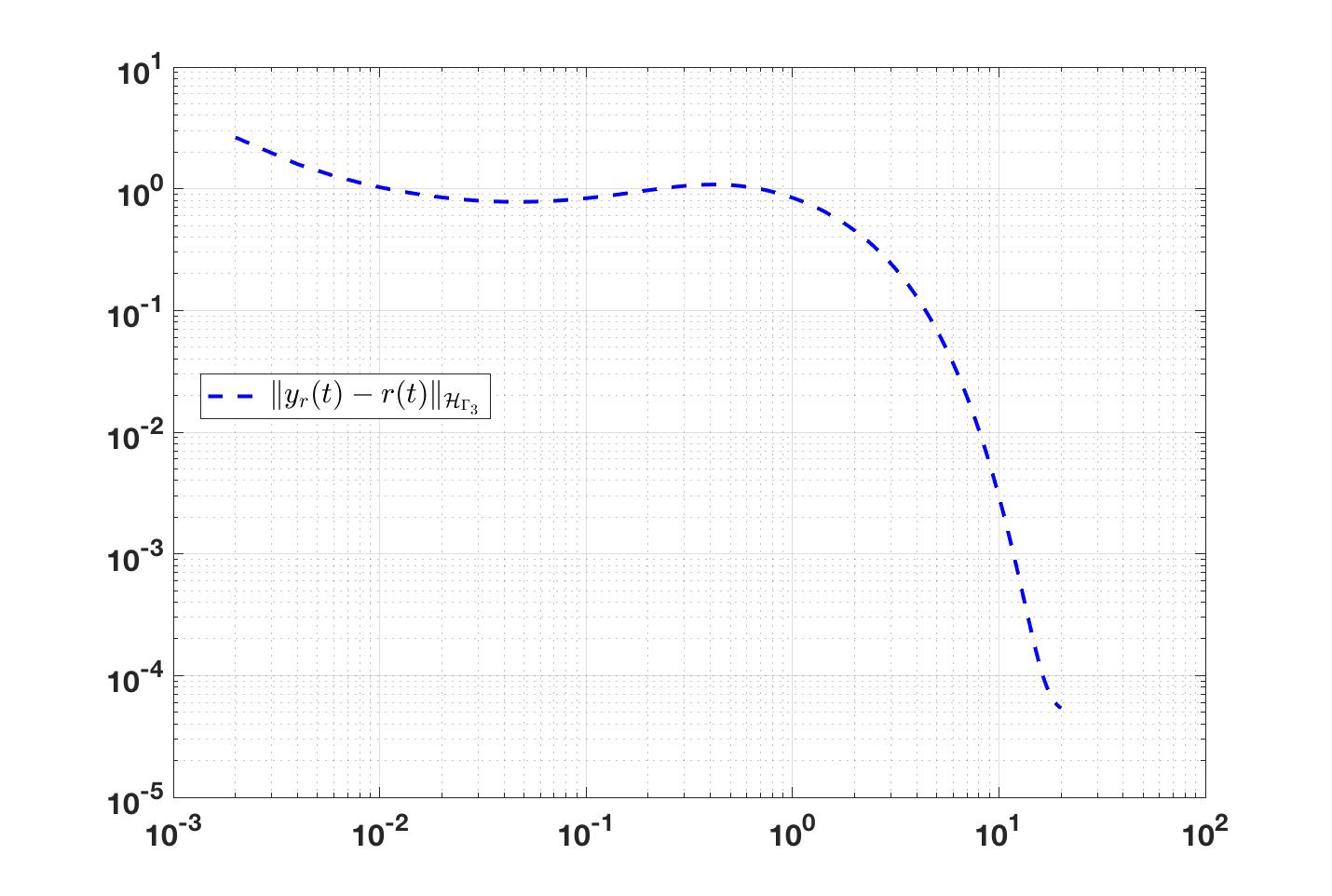}
	\caption{The tracking error $\|e(t)\|_{\mathcal{H}_{\Gamma_{3}}}=\|y_{r}(t)-r(t)\|_{\mathcal{H}_{\Gamma_{3}}}$}
	\label{figsum2}
\end{figure}

\section{Conclusion and comments}\label{section7}
In this paper, we present a mathematical analysis of a system of two dimensional advection diffusion equations coupled at the boundary .  This system of equations models  the  heat transfer in a direct contact membrane distillation process. We introduce new formulation of the problem, based on a semi group framework and provide the well-posedness criteria for the system, using the operator theory. We also analyze the co-current case for which the operator has been shown to be diagonalizable. However, based on the ADRC, we design a feedback law for the DCMD system, and show that the performance output exponentially tracks the reference signal. We then discuss  the robustness to the measurement noise. Moreover, a numerical example showing the effectiveness of the proposed control strategy is given.

In future studies, we plan to investigate  the ADRC for this kind of parabolic system in order to track the membrane  temperature, even though it will be a difficult task. In fact, by tracking desired temperature differences across the membrane, we will be able to improve the drinking water production of DCMD process. To do so, we need first to consider the ADRC control of the parabolic PDE when the control and reference tracking are in the distinct part of the boundary.
 
\bibliographystyle{siamplain}
\bibliography{paperGLV}

\begin{thebibliography}{10}

\bibitem{alsaadi2014experimental}
{\sc A.~Alsaadi, L.~Francis, G.~Amy, and N.~Ghaffour}, {\em Experimental and
  theoretical analyses of temperature polarization effect in vacuum membrane
  distillation}, Journal of Membrane Science, 471 (2014), pp.~138--148,
  \url{https://doi.org/10.1016/j.memsci.2014.08.005}.

\bibitem{boumenir2019monitoring}
{\sc A.~Boumenir, M.~Ghattassi, and T.~M. Laleg-Kirati}, {\em The
  reconstruction of a parabolic system}, Mathematical Methods in the Applied
  Sciences, 43 (2020), pp.~1399--1408.

\bibitem{Brezis}
{\sc H.~Brezis}, {\em Functional analysis, {S}obolev spaces and partial
  differential equations}, Universitext, Springer, New York, 2011.

\bibitem{cieslak2010finite}
{\sc T.~{Cie\'slak} and {Philippe Lauren\c ot}}, {\em Finite time blow-up for a
  one-dimensional quasilinear parabolic-parabolic chemotaxis system}, Annales
  de l'Institut Henri Poincar\'e (C) Non Linear Analysis, 27 (2010), pp.~437 --
  446, \url{https://doi.org/10.1016/j.anihpc.2009.11.016}.

\bibitem{corrias2004global}
{\sc L.~Corrias, B.~Perthame, and H.~Zaag}, {\em Global solutions of some
  chemotaxis and angiogenesis systems in high space dimensions}, Milan Journal
  of Mathematics, 72 (2004), pp.~1--28,
  \url{https://doi.org/10.1007/s00032-003-0026-x}.

\bibitem{desoer1985tracking}
{\sc C.~Desoer and C.~A. Lin}, {\em Tracking and disturbance rejection of mimo
  nonlinear systems with pi controller}, IEEE Transactions on Automatic
  Control, 30 (1985), pp.~861--867.

\bibitem{el2006framework}
{\sc M.~El-Bourawi, Z.~Ding, R.~Ma, and M.~Khayet}, {\em A framework for better
  understanding membrane distillation separation process}, Journal of membrane
  science, 285 (2006), pp.~4--29,
  \url{https://doi.org/10.1016/j.memsci.2006.08.002}.

\bibitem{eleiwi2016dynamic}
{\sc F.~Eleiwi, N.~Ghaffour, A.~S. Alsaadi, L.~Francis, and T.~M.
  Laleg-Kirati}, {\em Dynamic modeling and experimental validation for direct
  contact membrane distillation $\text{(DCMD)}$ process}, Desalination, 384
  (2016), pp.~1--11.

\bibitem{evans2010partial}
{\sc L.~C. Evans}, {\em Partial Differential Equations}, Graduate studies in
  mathematics, American Mathematical Society, second~ed., 2010.

\bibitem{Faierman}
{\sc M.~Faierman}, {\em Regularity of solutions of an elliptic boundary value
  problem in a rectangle}, Communications in Partial Differential Equations, 12
  (1987), pp.~285--305, \url{https://doi.org/10.1080/03605308708820493}.

\bibitem{feng2017new}
{\sc H.~Feng and B.~Z. Guo}, {\em A new active disturbance rejection control to
  output feedback stabilization for a one-dimensional anti-stable wave equation
  with disturbance}, IEEE Transactions on Automatic Control, 62 (2017),
  pp.~3774--3787.

\bibitem{francis1977linear}
{\sc B.~A. Francis}, {\em The linear multivariable regulator problem}, SIAM
  Journal on Control and Optimization, 15 (1977), pp.~486--505.

\bibitem{guo2016performance}
{\sc W.~Guo and B.~Z. Guo}, {\em Performance output tracking for a wave
  equation subject to unmatched general boundary harmonic disturbance},
  Automatica, 68 (2016), pp.~194--202.

\bibitem{guo2017adaptive}
{\sc W.~Guo, Z.~C. Shao, and M.~Krstic}, {\em Adaptive rejection of harmonic
  disturbance anticollocated with control in 1d wave equation}, Automatica, 79
  (2017), pp.~17--26.

\bibitem{jin2018performance}
{\sc F.~F. Jin and B.~Z. Guo}, {\em Performance boundary output tracking for
  one-dimensional heat equation with boundary unmatched disturbance},
  Automatica, 96 (2018), pp.~1--10.

\bibitem{karam2017analysis}
{\sc A.~M. Karam, A.~S. Alsaadi, N.~Ghaffour, and T.~Laleg-Kirati}, {\em
  Analysis of direct contact membrane distillation based on a lumped-parameter
  dynamic predictive model}, Desalination, 402 (2017), pp.~50--61.

\bibitem{khayet_book}
{\sc M.~Khayet}, {\em Advanced Membrane Technology and Applications}, John
  Wiley \& Sons, Inc., 2008, ch.~Membrane Distillation, pp.~297--369.

\bibitem{lee2017total}
{\sc J.-G. Lee, A.~S. Alsaadi, A.~M. Karam, L.~Francis, S.~Soukane, and
  N.~Ghaffour}, {\em Total water production capacity inversion phenomenon in
  multi-stage direct contact membrane distillation: A theoretical study},
  Journal of Membrane Science, 544 (2017), pp.~126--134,
  \url{https://doi.org/10.1016/j.memsci.2017.09.020}.

\bibitem{lions2012non}
{\sc J.~L. Lions and E.~Magenes}, {\em Non-homogeneous boundary value problems
  and applications}, vol.~1, Springer Science \& Business Media, 2012.

\bibitem{meurer2012control}
{\sc T.~Meurer}, {\em Control of Higher--Dimensional PDEs: Flatness and
  Backstepping Designs}, Springer Science \& Business Media, 2012.

\bibitem{meurer2009trajectory}
{\sc T.~Meurer and A.~Kugi}, {\em Trajectory planning for boundary controlled
  parabolic pdes with varying parameters on higher-dimensional spatial
  domains}, IEEE Transactions on Automatic Control, 54 (2009), pp.~1854--1868.

\bibitem{mizoguchi2014nondegeneracy}
{\sc N.~Mizoguchi and P.~Souplet}, {\em Nondegeneracy of blow-up points for the
  parabolic {K}eller-{S}egel system}, Annales de l'Institut Henri Poincar\'e
  (C) Non Linear Analysis, 31 (2014), pp.~851--875,
  \url{https://doi.org/10.1016/j.anihpc.2013.07.007}.

\bibitem{naidu2017transport}
{\sc G.~Naidu, W.~G. Shim, S.~Jeong, Y.~Choi, N.~Ghaffour, and S.~Vigneswaran},
  {\em Transport phenomena and fouling in vacuum enhanced direct contact
  membrane distillation: Experimental and modelling}, Separation and
  Purification Technology, 172 (2017), pp.~285--295,
  \url{https://doi.org/10.1016/j.seppur.2016.08.024}.

\bibitem{pazysemigroups}
{\sc A.~Pazy}, {\em Semigroups of Linear Operators and Applications to Partial
  Differential Equations}, Applied Mathematical Sciences, Springer New York,
  2012, \url{https://books.google.fr/books?id=DQvpBwAAQBAJ}.

\bibitem{shim2015solar}
{\sc W.~G. Shim, K.~He, S.~Gray, and I.~S. Moon}, {\em Solar energy assisted
  direct contact membrane distillation ({DCMD}) process for seawater
  desalination}, Separation and Purification Technology, 143 (2015),
  pp.~94--104,
  \url{https://doi.org/http://doi.org/10.1016/j.seppur.2015.01.028}.

\bibitem{TucWeiss}
{\sc M.~Tucsnak and G.~Weiss}, {\em Observation and Control for Operator
  Semigroups}, Birkh\" auser {A}dvanced {T}exts {B}asler {L}ehrb\" ucher,
  Birkh\" auser {B}asel, 2009, \url{https://doi.org/10.1007/978-3-7643-8994-9}.

\bibitem{zhou2018performance}
{\sc H.-C. Zhou and B.-Z. Guo}, {\em Performance output tracking for
  one-dimensional wave equation subject to unmatched general disturbance and
  non-collocated control}, European Journal of Control, 39 (2018), pp.~39--52.

\end{thebibliography}

\appendix

\section{The operator related to the membrane distillation system}
\label{deft-op-A}
 First, we introduce the  domain of the operator: consider
 the   space $\mathbf{E}$, which is the set of pairs 
 $(f,p)$ in  $\bigl[H^{1}(\Omega) \cap C^{1}(\bar\Omega)\bigr]^{2}$ such that
 \begin{itemize}
 \item
 $ f(x,0)=p(x,L) = 0$  for every $x\in(0,1)$\,;
 \item
 $\partial_{x} f(0,y)=\partial_{x} p(0,y) = 0$,   for every $ y \in (0,L)$\,;
 \item
$\partial_y f(x,L)=\partial_y p(x,0) = 0$,   for every $x\in(0,1)$\,;
\item
$ \partial_{x}f(1,y)= - \gamma_f\bigl(f(1,y)-p(1,y)\bigr)$, for every $y\in(0,L)$\,;
\item
$ \partial_{x}p(1,y)=\gamma_p\bigl(f(1,y)-p(1,y)\bigr)$ for every $y\in(0,L)$\,.
 \end{itemize}
The space $H^1( \Omega)\times H^1(\Omega)$ is equipped with the product topology, so the norm of an element 
$(f,p)\in H^1( \Omega)\times H^1( \Omega)$ is defined by
\begin{equation}
\label{def-norme-H1}
\|\left(f,p\right)\|^2= \int_ \Omega f^2 + \int_ \Omega \vert\nabla f \vert^2 +\int_ \Omega p^2 + \int_ \Omega \vert\nabla p\vert^2\,.
\end{equation}
On $L^2( \Omega)\times L^2( \Omega)$, we consider the following inner product:
 \begin{equation}
 \label{ps-L2}
  \Ps{(f,p)}{(g,q)}:= 
\alpha_{p}\gamma_p \int_ \Omega f(x,y)g(x,y)\,\d x\d y + \alpha_f \gamma_f\int_ \Omega p(x,y)q(x,y)\,\d x\d y\,;
 \end{equation}
 notice that this inner product induces the product topology on $L^2( \Omega)\times L^2( \Omega)$\,.  
We denote by $\Hbc$ the closure of $\mathbf{E}$ in $\bigl[H^{1}(\Omega)\bigr]^{2}$\,; from the Poincar\'e's inequality, the induced norm on $\Hbc$ defined by~\eqref{def-norme-H1} is equivalent to the following one
\begin{equation}
\label{def-norm-Hbc}
\|(f,p)\|^2_{\Hbc} =   \alpha_f\int_ \Omega \|\nabla f\|^2 + \alpha_p \int_ \Omega \|\nabla p\|^2, \quad(f,p)\in\Hbc\,.
\end{equation}
We then denote by $A_0$  the operator  whose domain is given by 
\[
\mathcal{D}(A_0):=
\ens{(f,p)\in \Hbc}{(\Delta f,\Delta p)\in \bigl[L^{2}(\Omega)\bigr]^{2} }
\]
and which is defined by
\[
A_0(f ,p ) = \bigl(\alpha_{f}\Delta f , \alpha_{p}\Delta p\bigr)\,. 
\]
 
We introduce also the operator ${B}_0$, whose    domain is the one of $ {A}_0$ and which is defined   by 
\[
B_0(f,p)= 
\bigl(-\beta_{f}\partial_{y} f ,\beta_{p}\partial_{y} p\bigr)\,;
\]
finally we define operator $ {A}$, related to system~\eqref{glv1.1} by 
\[
  {A}= {A}_0+ {B}_0\,.
\]

\subsection{The operator ${A}_{0}$}
\label{sec:op-A0}
In this section, we shall prove that the operator $ {A}_0$ is, as the laplacian operator, self-adjoint and m-dissipative. The proof is analogous to the one that shows these classical properties of the laplacian operator (see \emph{e.g.}~\cite{TucWeiss}). We begin by the following proposition.

 \begin{proposition} \label{proposition_embedding}
The embedding operator from $\Hbc$ to $\bigl[L^{2}(\Omega)\bigr]^{2}$ is compact.
\end{proposition}
\begin{proof}
We denote by $ {J}$ the embedding $\Hbc \hookrightarrow \bigl[L^{2}(\Omega)\bigr]^{2}$
 From the elementary theory of Fourier series we know
that the family $\left(\varphi_{\alpha}\right)_{\alpha\in \mathbb{N}^{2}}$ defined by
\[
\varphi_ \alpha=\frac{2}{\sqrt L}\sin(\alpha_{1}\pi\,x)\sin\Bigl(\frac{\alpha_{2}\pi}L\,y\Bigr),\quad \alpha:=(\alpha_1, \alpha_2)\in \mathbb{N}^{2},
\]
 is an orthonormal basis for $L^{2}( \Omega)$. In this proof, we need the notation ${\|\alpha\|^{2}=\alpha^{2}_{1}+\alpha^{2}_{2}}$
for a multi-index $\alpha\in\mathbb{N}^{2}$. Let $\left(f,p\right) \in \Hbc$. Then
\[
\|f\|^{2}_{L^{2}\left(\Omega\right)}=\sum_{\alpha\in\mathbb{N}^{2}}\Big|\left<f,\varphi_{\alpha}\right>\Big|^{2}
\]
and
\[
\|f\|^{2}_{H^{1}\left(\Omega\right)}=\sum_{\alpha\in\mathbb{N}^{2}}\left(1+ \|\alpha\|^{2}\right)\Big|\left<f,\varphi_{\alpha}\right>\Big|^{2}\,.
\]

From the above formulas it follows that if $m\in \mathbb{N}$, and if $ {J}_{m}\in\mathcal{L}(\Hbc,L^{2}(\Omega))$ is
defined by
\[
 {J}_{m}( f,p)=\sum_{\alpha\in\mathbb{N}^{2}, \|\alpha\|^{2}\leqslant m}\left<f,\varphi_{\alpha}\right>\varphi_{\alpha}+\sum_{\alpha\in\mathbb{N}^{2}, \|\alpha\|^{2}\leqslant m}\left<p,\varphi_{\alpha}\right>\varphi_{\alpha}\,,
\]
 then, as 
\[
\sum_{ \alpha\in\N^2}\| \alpha\|^2|\ps{f}{ \varphi_ \alpha}|^2 \ge m\sum_{\mathclap{\substack{ \alpha\in\N^2\\\| \alpha\|^2>m}}}|\ps f{ \varphi_ \alpha}|^2 \text{ and }
\sum_{ \alpha\in\N^2}\| \alpha\|^2|\ps{p}{ \varphi_ \alpha}|^2 \ge m\sum_{\mathclap{\substack{ \alpha\in\N^2\\\| \alpha\|^2>m}}}|\ps p{ \varphi_ \alpha}|^2\,,
\]
we have
\[
\| {J}\left(f,p\right)- {J}_{m}\left(f,p\right)\|^{2}_{L^{2}\left(\Omega\right)}\leqslant \frac{1}{1+m}\|\left(f,p\right)\|_{\Hbc}\,,
\]
and so
\[
\lim_{m\to\infty}  {J}_m= {J}\,.
\]
Since the dimension of $\ran( {J}_{m})$  is finite, this implies that $ {J}$ is compact (see \cite[Proposition 12.2.2]{TucWeiss}).
\end{proof}

\begin{thm}
The operator $ {A}_{0}$ is self-adjoint and diagonalizable.
\end{thm}
\begin{proof}
Take $(f,p)$ and $(g,q)$ in $\mathcal{D}( {A}_0)$; integrating two times by parts, we obtain 
\begin{align}
\notag
\int_ \Omega \diffp[2]{f}{x}(x,y)g(x,y)\d x \d y 
& \begin{multlined}[t]
=\int_0^L\diffp{f}{x}(1,y)\,g(1,y)\d y - \int_0^Lf(1,y)\,\diffp{g}{x}(1,y)\,\d y\\[0.5ex]
  +\int_ \Omega f(x,y)\, \diffp[2]{g}{x}(x,y)\,\d x\d y
\end{multlined}\\
\label{calcul-ps1}
&\begin{multlined}[b]
=\gamma_f\int_0^L\bigl(g(1,y)p(1,y) - f(1,y)q(1,y)\bigr)\,\d y \\
 +\int_ \Omega f(x,y)\diffp[2]{g}{x}(x,y)\,\d x\d y\,,
 \end{multlined}
\end{align}
and 
\begin{align}
\label{calcul-ps2}
\int_ \Omega \diffp[2]{f}{y}(x,y)g(x,y)\d x \d y & = \int_ \Omega f(x,y)\diffp[2]{g}{y}(x,y)\,\d x\d y\,.
\end{align}
Analogous computations lead to
\begin{align}
\label{calcul-ps3}
\int_ \Omega \Delta p(x,y)\,q(x,y)\d x \d y  
&\begin{multlined}[t]
=
\gamma_p\int_0^L\bigl(f(1,y)q(1,y) - g(1,y)p(1,y)  \bigr)\,\d y \\
 +\int_ \Omega p(x,y)\, \Delta q(x,y)\,\d x\d y\,.
\end{multlined}
\end{align}
So we have
\begin{align*}
\ps{ {A}_0(f,p)}{(g,q)} & = 
\begin{multlined}[t]
 \alpha_p \gamma_p \alpha_f \gamma_f\int_0^L\bigl(g(1,y)p(1,y) - f(1,y)q(1,y)\bigr)\,\d y \\
 + \alpha_p \gamma_p \alpha_f\int_ \Omega f(x,y)\, \Delta g(x,y)\,\d x\d y\\
+ \alpha_f \gamma_f \alpha_p \gamma_p\int_0^L\bigl(f(1,y)q(1,y) - g(1,y)p(1,y)  \bigr)\,\d y \\
+ \alpha_f \gamma_f \alpha_p\int_ \Omega p(x,y)\, \Delta q(x,y)\,\d x\d y
 \end{multlined}\\
&= 
\begin{multlined}[t]
\alpha_p \gamma_p\!\!\int_ \Omega f(x,y)\,\bigl( \alpha_f \Delta g(x,y)\bigr)\,\d x\d y \\
+ \alpha_f \gamma_f\!\! \int_ \Omega p(x,y)\,\bigl(\alpha_p \Delta q(x,y)\bigr)\,\d x\d y
\end{multlined}\\
= &\Ps{(f,p)}{ {A}_0(g,q)}\,,
\end{align*}
which shows that the operator  $ {A}_{0}$ is symmetric. 
We shall now show that $ {A}_{0}$ is self-adjoint; to do this, it is enough to prove that $ {A}_{0}$ is onto, the proof is almost the same as in~\cite[Proposition 3.2.4 ]{TucWeiss}.

Take $(u,v) \in \bigl[L^{2}(\Omega)\bigr]^{2}$,  we have to prove the existence of  
$(f,p)$ in $ \mathcal{D}( {A}_{0})$ such that 
\[
 {A}_{0}\left(f,p\right)=\left(u,v\right).
\]
First notice that the mapping 
\[
(g,q)	\longrightarrow \int_{\Omega}u\,g +\int_{\Omega}v\,q
\]
 is a bounded linear functional on $\Hbc$. By the Riesz representation theorem, there exists $(f,p)\in\Hbc$ such that
\[
\Ps{(f,p)}{(g,q)}_{\Hbc}=\Ps{(u,v)}{(g,q)}_{\left[L^{2}(\Omega)\right]^{2}}\,;
\]
the inner product in the left-hand member of this equality being the one related to the norm defined by~\eqref{def-norm-Hbc}.
Denoting by $\mathcal{D}( \Omega)$ the space of smooth functions with compact support in $ \Omega$,  we notice  that, as $\mathcal{D}( \Omega)\times\mathcal{D}( \Omega)\subset\Hbc$, the above equality can also be written for any $g,q\in\mathcal{D}( \Omega)$; so, in the sense of distributions,  we have
\begin{equation*}
\begin{aligned}
\Ps{(f,p)}{(g,q)}_{\Hbc} =&\alpha_f\int_ \Omega(\nabla f)\cdot(\nabla g) \,\d x\d y+ \alpha_p\int_ \Omega(\nabla p)\cdot(\nabla q)\,\d x\d y \\=&
- \alpha_f\ps{ \Delta f}{ g}_{\mathcal{D}',\mathcal{D}} - \alpha_p\ps{ \Delta p}{ q}_{\mathcal{D}',\mathcal{D}}\,.
\end{aligned}
\end{equation*}
This equality is true for every $(g,q)\in[\mathcal{D}( \Omega)]^2$, so we have
\begin{align*}
 -\alpha_{f}\Delta f&=u,  \text{ in $ \mathcal{D}^{'}(\Omega)$}, &  -\alpha_{p}\Delta p&=v, \text{ in $ \mathcal{D}^{'}(\Omega)$}.
\end{align*}
Since, $(u,v) \in \left[L^{2}(\Omega)\right]^{2}$, we obtain that 
\begin{align*}
\alpha_{f}\Delta f& \in L^{2}\left(\Omega\right),&\alpha_{p}\Delta p&\in L^{2}\left(\Omega\right). 
\end{align*}
Thus $\left(f,p\right) \in \mathcal{D}( {A}_{0})$ and 
\[
 {A}_{0}\left(f,p\right)=\left(u,v\right),
\]
hence $ {A}_{0}$ is onto and  we can conclude that $ {A}_{0}$ is a self adjoint.

Finally, according to Proposition \ref{proposition_embedding}, the embedding  $ {J}:\Hbc \hookrightarrow {[L^{2}(\Omega)]}^2$ is compact, therefore $  {A}^{-1}_{0}= {J}\circ  {A}_0^{-1}$ is compact and hence, by Proposition \cite[Theorem 3.2.12]{TucWeiss}, $ {A}_{0}$ is diagonalizable with an orthonormal basis $(\varphi_{k},\psi_k)$ of eigenvectors and the corresponding sequence of eigenvalues $\left(\lambda_{k}\right)$ satisfies and $\lim_{|k|\to\infty} \lambda_{k}= \infty$.
\end{proof}

\subsection{The  operator $ {A}$}
In this section, we shall prove that $ {A}$ is m-dissipative with respect to the inner product defined in~\eqref{ps-L2}; the proof is in the same spirit  as the proof of~\cite[Theorem 3.2]{pazysemigroups}.
We shall see first that $ {A}$ is dissipative.
\begin{proposition}
\label{prop-dissip-1}
For every $t\in[0,1]$, the operator $ {A}_{0}+ t {B}_0$ is  dissipative; moreover, the operator $ {A}_0$ is m-dissipative.
\end{proposition}
\begin{proof}
Take $(f,p)\in\mathcal{D}( {A}_0)$, we compute first $\Ps{ {A}_0(f,p)}{(f,p)}$: integrating by parts, we get
\begin{align*}
\int _ \Omega \diffp[2]{f}{x}f &= \int_0^L\left[\diffp{f}{x}(x,y)\,f(x,y)\right]_{x=0}^{x=1}\,\d y
 - \int_ \Omega \left(\diffp{f}{x}\right)^2 \negthickspace\d x\d y\\
& = -\gamma_f \int_0^Lf(1,y)\bigl(f(1,y)-p(1,y)\bigr)\,\d y  - \int_ \Omega \left(\diffp{f}{x}\right)^2\negthickspace\d x\d y\,.
\end{align*}
On the other hand
\begin{align*}
\int _ \Omega \diffp[2]{f}{y}f &= \int_0^1\left[\diffp{f}{y}(x,y)\,f(x,y)\right]_{y=0}^{y=1}\d x - \int_ \Omega \left(\diffp{f}{y}\right)^2\negthickspace\d x\d y\\
& = - \int_ \Omega \left(\diffp{f}{x}\right)^2\negthickspace\d x\d y\,.
\end{align*}
An analogous computation leads to
\begin{align*}
\int_ \Omega (\Delta  p)p\, \d x\d y&= \gamma_p\int_0^Lp(1,y)\bigl(f(1,y)-p(1,y)\bigr)\d y  - \int_ \Omega \vert\nabla p\vert^2\d x\d y
\end{align*}
So, we have
\begin{align}
\notag
\Ps{ {A}_0(f,p)}{(f,p)}_{{[L^2( \Omega)]}^2} & =\alpha_p \gamma_p \alpha_f \int_ \Omega ( \Delta f)f\,\d x\d y
 + \alpha_f \gamma_f \alpha_p\int_  \Omega ( \Delta p)p\,\d x\d y\\
  \notag &
   \begin{multlined}[t]
   =
 - \alpha_p \gamma_p \alpha_f\gamma_f \int_0^Lf(1,y)\bigl(f(1,y)-p(1,y)\bigr)\,\d y\\
 -  \alpha_p \gamma_p \alpha_f\int_ \Omega|\nabla f|^2\,\d x\d y\\
  + \alpha_f \gamma_f \alpha_p \gamma_p\int_0^Lp(1,y)\bigl(f(1,y)-p(1,y)\bigr)\,\d y\\ 
   -  \alpha_f \gamma_f \alpha_p\int_ \Omega|\nabla p|^2\,\d x\d y
   \end{multlined}\\
   &
   \begin{multlined}[t]
   \label{dissp-A0}
   =  - \alpha_p \gamma_p \alpha_f\gamma_f \!\!\int_0^L\!\!\!\!\left(f(1,y)-p(1,y)\right)^2\d y\\[0.75em]
   - \alpha_f \alpha_p\!\!\int_ \Omega\bigl( \gamma_p|\nabla f|^2+ \varphi_f|\nabla p|^2\bigr)\d x\d y
   \end{multlined}
\end{align}
On the other  hand 
\begin{align}
\notag
\ps{ {B}_0(f,p)}{(f,p)} &= - \alpha_p \gamma_p \beta_f\int_ \Omega (\partial_y f)f\,\d x\d y+
 \alpha_f \gamma_f \beta_p\int_ \Omega (\partial_y p)p\,\d x\d y\\
\notag &
\begin{multlined}[t]
=  -\frac{\alpha_p \gamma_p \beta_f}2\,\int_0^1\bigl(f^2(x,1)- f^2(x,0)\bigr)\,\d x \\
+\frac{  \alpha_g \gamma_f \beta_p}2\,\int_0^1\bigl(p^2(x,1)- p^2(x,0)\bigr)\,\d x
\end{multlined}\\
\label{dissip-B0}
&= -\frac{ \alpha_p \gamma_p \beta_f}2\,\int_0^1f^2(x,1)\,\d x -\frac{ \alpha_f \gamma_f \beta_p}2\,\int_0^1 p^2(x,0)\,\d x\,.
\end{align}
Inequalities~\eqref{dissp-A0} and~\eqref{dissip-B0} prove that $ {A}_0+t {B}_0$ is dissipative for every $t\in[0,1]$. We have also seen that $ {A}_0$ is onto, which proves that   $ {A}_0$ is self-adjoint, as $ {A}_0$ is dissipative, we can conclude that 
$ {A}_0$ is m-dissipative. 
 \end{proof}
 Now, as in the proof of the previous theorem, we shall prove  that there exists $ \delta>0$ such that, if 
 $ {A}_0+t_0 {B}_0 $ (here $t_0\in[0,1]$) is m-dissipative, then $ {A}_0+t  {B}_0$ is also m-dissipative for every $t\in[0,1]$ such that $ |t-t_0|\le \delta$. To this end, we need the following proposition.
 \begin{proposition}
 \label{prop-dissip-2}
 If the operator $ {A}_0+t_0 {B}_0$ ($ t_0\in[0,1]$) is m-dissipative , then the operator
  $ {B}_0 \bigl(\Id - ( {A}_0+t_0 {B}_0)\bigr)^{-1}$ is bounded, the bound being independent from $t_0$\,.
 \end{proposition}

\begin{proof}
Notice that if $ {A}_0+t_0 {B}_0$  is m-dissipative, $\Id - ( {A}_0+t_0 {B}_0)$ is invertible. We denote the inverse 
$\bigl(\Id - ( {A}_0+t_0 {B}_0)\bigr)^{-1}$ by $R(t_0)$ and we take $(f,p)$ in ${[L^2( \Omega)]}^2$. We seek for an upper bound for $\| {B}_0R(t_0)(f,p)\|$. Let $(u,v)=R(t_0)(f,p)$, so that we have
\begin{align}
\label{rel-(f-p)-(u-v)}
f & = u- \alpha_f \Delta u +t_0 \beta_f\,\diffp{ u}{y} , & p & = v - \alpha_p \Delta v -t_0 \beta_p\,\diffp{ v}{y}\,.
\end{align}
Now
\begin{align}
\notag
\| {B}_0R(t_0)(f,p)\|^2 & = \| {B}_0(u,v)\|^2\\
\notag
& = \alpha_p \gamma_p\int_ \Omega \beta_f^2\left(\diffp{ u}{ y}\right)^2\negthickspace \d x\d y
+ \alpha_f \gamma_f\int_ \Omega \beta_p^2\left(\diffp{ p}{ y}\right)^2\negthickspace\d x\d y\\[0.75ex]
\label{prop3-ineg0}
&\le M\left(  \gamma_p\int_ \Omega \left(\diffp{ u}{ y}\right)^2\negthickspace\d x\d y +  \gamma_f\int_ \Omega \left(\diffp{ p}{ y}\right)^2\negthickspace\d x\d y\right)
\end{align}
with $M=\max( \alpha_p \beta_f^2, \alpha_f \beta_p^2)$. 

We shall  rewrite these two integrals; first we have
\begin{align}
\notag
\int_ \Omega \left(\diffp{ u}{y}\right)^2\negthickspace  \d x\d y& = \int_0^1\left[u(x,y)\,\diffp{u}{y}(x,y) \right]_{y=0}^{y=L}\d x 
- \int_ \Omega u\,\diffp[2]{u}{y}\,\d x\d y\\[1ex]
\notag
& = - \int_ \Omega u\,\diffp[2]{u}{y}\,\d x\d y\\[1ex]
\label{prop3-eg1}
& = -\frac1{ \alpha_f}\int_ \Omega u\left(u - \alpha_f\diffp[2]{u}{x} + t_0 \beta_f\diffp{u}{y} - f\right)\d x\d y &\text{from~\eqref{rel-(f-p)-(u-v)}}\,.
\end{align}
Now, we have
\begin{align}
\notag
\int_ \Omega u\diffp[2]{u}{x} \,\d x\d y&  = \int_0^L\left[ u(x,y)\diffp{u}{x}(x,y)\right]_{x=0}^{x=1}\d y - \int_ \Omega\left( \diffp{u}{x}\right)^2 \d x \d y\\
\label{prop3-eg2}
& = - \gamma_f\int_0^L u(1,y)\bigl(u(1,y)-v(1,y)\bigr)\,\d y - \int_ \Omega\left( \diffp{u}{x}\right)^2\negthickspace\d x\d y\,;
\end{align}
on the other hand
\begin{align}
\label{prop3-eg3}
\int_ \Omega u\diffp{u}{y}\,\d x\d y & = \frac12\,\int_0^1u^2(x,1)\,\d x\,.
\end{align}
Substituting equalities~\eqref{prop3-eg2} and~\eqref{prop3-eg3}  into~\eqref{prop3-eg1}, we get
\begin{equation}  \label{prop3-ineg1}
\begin{aligned}
\int_ \Omega \left(\diffp{ u}{y}\right)^2\negthickspace\d x\d y=
&-\frac1{ \alpha_f}\,\int_ \Omega u^2 \d x\d y
- \gamma_f\int_0^L u(1,y)\bigl(u(1,y)-v(1,y)\bigr)\,\d y \\
 & - \int_ \Omega\left( \diffp{u}{x}\right)^2\negthickspace \d x\d y- \frac{t_0 \beta_f}{2 \alpha_f}\int u^2(x,L)\,\d x +\frac1{ \alpha_f}\,\int_ \Omega uf\, \d x\d y\\
\le&-\frac1{ \alpha_f}\,\int_ \Omega u^2 \d x\d y- \gamma_f\int_0^L u(1,y)\bigl(u(1,y)-v(1,y)\bigr)\,\d y\\  
  &+\frac1{ \alpha_f}\,\int_ \Omega u\,f \d x\d y\\
 \le&-\frac1{ \alpha_f}\,\int_ \Omega u^2 \d x\d y- \gamma_f\int_0^L u(1,y)\bigl(u(1,y)-v(1,y)\bigr)\,\d y \\
 &+\frac1{ 2\alpha_f}\,\int_ \Omega (u^2 +f^2) \d x\d y\\
 \le &-\gamma_f\int_0^L u(1,y)\bigl(u(1,y)-v(1,y)\bigr)\,\d y + \frac{1}{2 \alpha_f}\int_ \Omega f^2 \d x\d y\,.
\end{aligned}
\end{equation}
An analogous computation shows that
\begin{align}
\label{prop3-ineg2}
\int_ \Omega \left(\diffp{ v}{y}\right)^2 \negthickspace\d x\d y
&\le 
  \gamma_p\int_0^L v(1,y)\bigl(u(1,y)-v(1,y)\bigr)\,\d y  
  +\frac1{2\alpha_p}\,\int_ \Omega p^2 \d x\d y \,.
\end{align}
Substituting~\eqref{prop3-ineg1} and~\eqref{prop3-ineg2} into~\eqref{prop3-ineg0}, we get
\begin{align*}
\|  {B}_0&R(t_0)(f,p)\|^2 \le 
M\gamma_p\left(  - \gamma_f\int_0^L u(1,y)\bigl(u(1,y)-v(1,y)\bigr)\,\d y + \frac{1}{2 \alpha_f}\int_ \Omega f^2 \d x\d y \right) \\
&+ M\gamma_f\left( \gamma_p\int_0^L v(1,y)\bigl(u(1,y)-v(1,y)\bigr)\,\d y  +\frac1{2\alpha_p}\,\int_ \Omega p^2 \d x\d y\right)\\
&= -M \gamma_f \gamma_p\int_0^L\bigl(u(1,y)-v(1,y)\bigr)^2\d y + 
M\left( \frac{\gamma_p}{2 \alpha_f}\int_ \Omega f^2 \d x\d y+ \frac{ \gamma_f}{2 \alpha_p}\,\int_ \Omega p^2 \d x\d y\right)\\
& \le M'\|(f,p)\|^2_{[L^2( \Omega)]^2},
\end{align*}
where $M' =\max \biggl(\dfrac M{2\alpha_f^2}, \dfrac M{2 \alpha_p^2}\biggr)$. 

We have proved that the operator $ {B}_0R(t_0)$ is bounded, moreover its norm is less than or equal to $\sqrt{M'}$, which is a bound independent from $t_0$.
\end{proof}
We are now ready to prove the  main result of this section.
\begin{thm}\label{A-m-dissipative}
Operator $ {A}$ is m-dissipative with respect to the inner product~\eqref{ps-L2}.
\end{thm}
\begin{proof}
Assume that $ {A}_0+t_0 {B}_0$ is m-dissipative (with $t_0\in[0,1]$), a simple computation shows that we can write
\begin{align}
\notag
\Id -( {A}_0+t {B}_0) & = \Id -( {A}_0+t_0 {B}_0) +(t_0-t) {B}_0\\
\label{A0+tB0}
& = \bigl(\Id+(t-t_0) {B}_0R(t_0)\bigr) \bigl(\Id - ( {A}_0+t_0 {B}_0)\bigr)\,.
\end{align}
The bounded operator $\Id+(t-t_0) {B}_0R(t_0)$ is invertible if $\|(t-t_0) {B}_0R(t_0)\|<1$, and this inequality is true if $|t-t_0|<1/\sqrt{M'}$. Thus, if $t$ satisfies this inequality, we can conclude from~\eqref{A0+tB0} that $\Id - ( {A}_0+t  {B}_0)$ is invertible. Moreover, as from Proposition~\ref{prop-dissip-1}, we know that 
$ {A}_0+t  {B}_0 $ is dissipative (for every $t\ge0$), we can conclude  that $ {A}_0+ t  {B}_0$ is m-dissipative for every $t\in[0,1]$ and such that 
$|t-t_0|\le 1/\sqrt{M'}$\,. Observe now that $ {A}_0+t_0  {B}_0$ is m-dissipative with $t_0=0$ (\emph{cf} Prop.~\ref{prop-dissip-1}), since any point of $[0,1]$ can be reached  from 0 by a finite number of step of length $1/\sqrt{M'}$, we conclude that 
$ {A}= {A}_0+ {B}_0$ is m-dissipative.
\end{proof}

\end{document}